\documentclass[a4paper,12pt]{article}
\usepackage{latexsym,amsmath,amsthm,amssymb}
\usepackage{a4wide}
\usepackage{hyperref}
\usepackage{marginnote}
\usepackage{color}
\hypersetup{
pdftitle={Some trace Hardy type inequalities and trace Hardy-Sobolev-Maz'ya type inequalities},    
pdfauthor={Van Hoang Nguyen},
colorlinks = true,
linkcolor = magenta,
citecolor = blue,
}

\theoremstyle{plain}
\newtheorem{theorem}{Theorem}[section]
\newtheorem{proposition}[theorem]{Proposition}
\newtheorem{lemma}[theorem]{Lemma}
\newtheorem{corollary}[theorem]{Corollary}

\theoremstyle{definition}

\theoremstyle{remark}

\renewcommand{\thefootnote}{\arabic{footnote}}

\def\R{\mathbb R}
\def\Z{\mathbb Z}

\def\al{\alpha}
\def\om{\omega}
\def\Om{\Omega}
\def\be{\beta}
\def\ga{\gamma}
\def\ge{\geq}
\def\de{\delta}
\def\De{\Delta} 
\def\Gam{\Gamma}
\def\si{\sigma}
\def\lam{\lambda}
\def\vphi{\varphi}
\def\ep{\epsilon}
\def\na{\nabla}
\def\pa{\partial}
\def\la{\langle} 
\def\ra{\rangle} 
\def\lt{\left}
\def\rt{\right}
\def\o{\overline}

\def\mC{\mathcal{C}}

\def\mH{\mathcal H}
\def\dHa{\dot{H}^\al(\R^n)}
\def\dHs{\dot{H}^{\frac{1-s}{2}}(\R^n)}

\def\dWds{\dot{W}(d_\mC^s,\mC)}
\def\hs{\R^{n+1}_+}
\def\sp{\R^n}

\def\spp{\R^{n+1}}
\def\d{\mathrm{div}}

\def\Rin{R_{\mathrm{in}}}
\def\OR{\Om\times \R}
\def\ORn{\R^n_+\times \R}
\def\COR{C_0^\infty(\OR)}
\def\CORn{C_0^\infty(\ORn)}
\def\Rn{\R^n_+}
\def\i0i{\int_0^\infty}

\numberwithin{equation}{section}

\title{Some trace Hardy type inequalities and trace Hardy-Sobolev-Maz'ya type inequalities}
\author{Van Hoang Nguyen\footnote{
School of Mathematical Sciences, Tel Aviv University, Tel Aviv 69978, \textsc{Israel}.}}

\begin{document}
\maketitle


\renewcommand{\thefootnote}{}

\footnote{Email: vanhoang0610@yahoo.com}

\footnote{Supported by a grant from the European Research Council (grant number $305629$)}

\footnote{2010 \emph{Mathematics Subject Classification\text}: 26D10, 46E35.}

\footnote{\emph{Key words and phrases\text}: Trace Hardy type inequality, trace Hardy-Sobolev-Maz'ya type inequality, logarithmic Hardy trace inequality, logarithmic Sobolev trace inequality.}

\renewcommand{\thefootnote}{\arabic{footnote}}
\setcounter{footnote}{0}

\begin{abstract}
We prove a trace Hardy type inequality with the best constant on the polyhedral convex cones which generalizes recent results of Alvino et al. and of Tzirakis on the upper half space. We also prove some trace Hardy-Sobolev-Maz'ya type inequalities which generalize the recent results of Filippas et al.. In applications, we derive some Hardy type inequalities and Hardy-Sobolev-Maz'ya type inequalities for fractional Laplacian. Finally, we prove the logarithmic Sobolev trace inequalities and logarithmic Hardy trace inequalities on the upper half spaces. The best constants in these inequalities are explicitly computed in the radial case.

\end{abstract}

\section{Introduction}
Let $n\geq 2$. The Hardy inequality on the upper half space $\hs =\{(x,t)\in \R^{n+1}\, :\, x\in \sp,\, t >0\}$ says that
\begin{equation}\label{eq:Hardy}
\int_{\hs}|\na u(x,t)|^2 dxdt \geq \frac{(n-1)^2}4 \int_{\hs} \frac{u(x,t)^2}{|x|^2+t^2} dx dt,
\end{equation} 
for any function $u\in C_0^\infty(\R^{n+1})$. By the density, the inequality \eqref{eq:Hardy} still holds for all function $u$ in $\dot{W}^{1,2}(\R^{n+1}_+)$ which is the completion under the norm $\|u\|_{\dot{W}^{1,2}(\R^{n+1}_+)} = (\int_{\hs}|\na u(x,t)|^2 dx dt)^{1/2}$ of the space of functions which are restriction to $\hs$ of functions in $C_0^\infty(\R^{n+1})$. The constant $(n-1)^2/4$ is sharp and never attains in $\dot{W}^{1,2}(\hs)$.

Another important inequality is the Hardy trace inequality on the upper half space $\hs$ (or Kato inequality) which asserts that 
\begin{equation}\label{eq:Kato}
\int_{\hs}|\na u(x,t)|^2 dxdt \geq 2 \lt(\frac{\Gamma(\frac{n+1}4)}{\Gamma(\frac{n-1}4)}\rt)^2 \int_{\pa\hs}\frac{u(x,0)^2}{|x|} dx,
\end{equation}
for any function $u \in \dot{W}^{1,2}(\hs)$. The constant $2 \Gamma((n+1)/4)^2/\Gamma((n-1)/4)^2$ also is sharp and never attains in $\dot{W}^{1,2}(\hs)$.

Recently, Alvino et al.  have proved an interesting inequality in \cite{AVF12} which interpolates between \eqref{eq:Hardy} and \eqref{eq:Kato}. This inequality states that for any $2\leq \beta \leq n+1$, it holds
\begin{align}\label{eq:AVF}
\int_{\hs}|\na u(x,t)|^2 dxdt & \geq \frac{(\be-2)^2}4 \int_{\hs} \frac{u(x,t)^2}{|x|^2+t^2} dx dt  + H(n,\beta)\int_{\pa\hs}\frac{u(x,0)^2}{|x|} dx,
\end{align}
for any function $u\in \dot{W}^{1,2}(\hs)$, with
\[
H(n,\beta) = 2 \frac{\Gamma(\frac{n+\be-1}4)\Gamma(\frac{n-\be+3}4)}{\Gamma(\frac{n+\be-3}4)\Gamma(\frac{n+1-\beta}4)}.
\]
Again, the constant $H(n,\beta)$ is sharp and never attains in $\dot{W}^{1,2}(\hs)$. The inequality \eqref{eq:AVF} was recently generalized by Tzirakis in \cite{T15}.

Our first aim of this paper is to extend the inequality \eqref{eq:AVF} for any polyhedral cone convex $\mC$ defined by
\begin{equation}\label{eq:cone}
\mC =\lt\{x\in \R^{n+1}\, :\, \la x, u_i\ra > 0,\quad i = 1,\ldots, m\rt\},
\end{equation}
where $u_1,\ldots, u_m$ are unit vectors in $\R^{n+1}$, $m\geq 1$. For $x\in \mC$, let us denote $d_\mC(x)$ the distance from $x$ to the boundary of $\mC$, i.e,
$$d_\mC(x) =\text{dist}(x,\pa\mC) =\min_{1\leq i\leq m} \la x,u_i\ra.$$
Given $s\in (-1,1)$, we define the weighted Sobolev space $\dWds$ to be the completion of $C_0^\infty(\o{\mC})$ under the norm
$$\|u\|_{\dWds} = \lt(\int_\mC |\na u(x)|^2 d_\mC(x)^s dx\rt)^{\frac12}.$$
We say that a function $u$ belongs to $C_0^\infty(\o{\mC})$ if it is a restriction of a compactly supported smooth function on $\spp$ to $\mC$. Our extension of \eqref{eq:AVF} to the polyhedral convex cone $\mC$ is as follows
\begin{theorem}\label{maintheorem}
Let $n\geq 2$, $s \in (-1,1)$, then for any $2\leq \be < n_s:=n+1+s$, there exists a constant $H(n,s,\be)$ such that
\begin{equation}\label{eq:maininequality}
\int_\mC |\na u|^2 d_\mC^s dx \geq \frac{(\beta-2)^2}4 \int_\mC \frac{|u(x)|^2}{|x|^2} d_\mC(x)^s dx + H(n,s,\be) \int_{\pa \mC}\frac{u(x)^2}{|x|^{1-s}} d\mH^n(x),
\end{equation}
where
\begin{equation}\label{eq:constantH}
H(n,s,\be) = 2\frac{\Gam\lt(\frac{1+s}2\rt) \Gam\lt(\frac{n_s+\be-2-2s}4\rt)\Gam\lt(\frac{n_s-\be+2-2s}4\rt)}{\Gam\lt(\frac{1-s}2\rt)\Gam\lt(\frac{n_s+\be-4}4\rt)\Gam\lt(\frac{n_s-\be}4\rt)},
\end{equation}
for any $u\in \dWds$, where $\mH^n$ is the $n-$dimensional Hausdorff measure on $\pa \mC$. Moreover, the constant $H(n,s,\be)$ is optimal.
\end{theorem}

Note that when $\mC$ is the upper half space $\hs$, Theorem \ref{maintheorem} recovers the recent results of Tzirakis \cite[Theorem 1]{T15} and Alvino et al. \eqref{eq:AVF}. The endpoint case $\be =2$, Theorem \ref{maintheorem} gives us a weighted trace Hardy inequality on $\mC$ as follows 
\begin{equation}\label{eq:traceHardy}
\int_{\mC} |\na u|^2 d_\mC^{s} dx \geq 2 \frac{\Gamma\lt(\frac{1+s}2\rt)}{\Gam\lt(\frac{1-s}2\rt)}\lt(\frac{\Gamma\lt(\frac{n+1-s}{4}\rt)}{\Gamma\lt(\frac{n-1+s}{4}\rt)}\rt)^2\,\int_{\pa\mC} \frac{u(x)^2}{|x|^{1-s}} d\mH^n(x), \quad u\in \dWds.
\end{equation}
Since $\lim_{\be\to n_s}H(n,s,\be) = 0$, then letting $\beta\to n_s$ yields a weighted Hardy inequality on $\mC$
\begin{equation}\label{eq:Hardyonhalfspace}
\frac{(n_s-2)^2}4 \int_{\mC} \frac{u(x)^2}{|x|^2} d_\mC(x)^s dx \leq \int_{\mC} |\na u(x)|^2 d_\mC(x)^s dx,\quad u\in \dWds.
\end{equation}
Again, the constant ${(n_s-2)^2}/4$ in \eqref{eq:Hardyonhalfspace} is sharp.  


For $\al\in (0,1)$, the fractional Laplacian operator $(-\De)^\al$ is defined by 
$$(-\De)^\al f(x) = \frac{\al 2^{2\al} \Gamma\lt(\frac{n+2\al}{2}\rt)}{\Gamma(1-\al) \pi^{\frac{n}{2}}}\, \mathrm{PV}\int_{\sp}\frac{f(x) -f(y)}{|x-y|^{n+2\al}}dy,$$
where $\mathrm{PV}$ stands for the Cauchy principle value integral. In a very remarkable paper \cite{CS07}, Caffarelli and Silvestre gave an equivalent definition for operator $(-\De)^\al$, $\al\in (0,1)$ via the Dirichlet to Neumann map by considering an extension problem in one more dimension in terms of a degenerate elliptic equation (see \cite{Yang14} for a recent interesting result concerning to the higher order extension for fractional Laplacian). The extension of Caffarelli and Silvestre is as follows. For any function $f$ on $\R^n$, let us consider the extension problem on $\hs$ given by
\begin{equation}\label{eq:extensionproblem}
\d(t^{1-2\al}\na_{(x,t)}u(x,t)) =0, \quad (x,t) \in \hs; \qquad u(x,0) = f(x).
\end{equation}
We recall that the solution of \eqref{eq:extensionproblem} minimizes the energy functional defined by
$$J[u] = \int_{\hs} |\na u(x,t)|^2 t^{1-2\al} dt dx,$$
over all functions $u$ satisfying the condition $u(x,0)= f(x)$. Moreover, if $u$ is a such solution then
$$(-\De)^{\al} f(x) = -2^{2\al-1}\frac{\Gamma(\al)}{\Gamma(1-\al)}\lim\limits_{t\to 0^+} t^{1-2\al} \pa_t u(x,t).$$
Hence, for any $u\in \dot{W}(d_{\hs}^s,\hs)$ we have
\begin{equation}\label{eq:traceonboundary}
\int_{\hs} |\na u(x,t)|^2 t^{s} dx dt \geq 2^{s} \frac{\Gamma\lt(\frac{1+s}{2}\rt)}{\Gam\lt(\frac{1-s}{2}\rt)}\, \|u(\cdot,0)\|_{\dHs}^2,
\end{equation}
where $\dHs$ is the homogeneous Sobolev space of order $(1-s)/2$ that is defined as the completion of $C_0^\infty(\R^n)$ under the norm
$$\|u\|_{\dHs} = \lt(\frac{1}{(2\pi)^n} \int_{\R^n} |\xi|^{1-s} |\hat{u}(\xi)|^2 d\xi\rt)^{1/2},\qquad u\in C_c^\infty(\R^n),$$
where
$$\hat{u}(\xi) = \int_{\R^n} e^{-i\la\xi, x\ra} u(x) dx,$$
is the Fourier transform of $u$. Equality occurs in \eqref{eq:traceonboundary} if $u$ solves the equation \eqref{eq:extensionproblem}. 

Theorem \ref{maintheorem} and \eqref{eq:traceonboundary} immediately imply the following Hardy type inequality for the fractional Laplacian on the upper half space,
\begin{corollary}
Given $n\geq 2$, $s\in (-1,1)$, $\be \in [2,n_s]$, $f \in \dHs$, and let $u$ be the solution of the equation \eqref{eq:extensionproblem} with $\al =(1-s)/2$. Then, it holds 
\begin{align*}
\|f\|_{\dHs}^2 &\geq 2^{-s}\frac{\Gam\lt(\frac{1-s}{2}\rt)}{\Gam\lt(\frac{1+s}{2}\rt)} \frac{(\be-2)^2}{4} \int_{\R^{n+1}_+} \frac{u(x,t)^2}{|x|^2 +t^2} t^{s} dt dx\\
&\quad \quad\quad + 2^{1-s}\frac{\Gam\lt(\frac{n_s+\be-2-2s}4\rt)\Gam\lt(\frac{n_s-\be+2-2s}4\rt)}{\Gam\lt(\frac{n_s+\be-4}4\rt)\Gam\lt(\frac{n_s-\be}4\rt)} \int_{\R^n} \frac{f(x)^2}{|x|^{1-s}} dx.
\end{align*}
\end{corollary}
For $\be =2$, we get the inequality
\begin{equation}\label{eq:fracHardyineq}
\|f\|^2_{\dHs} \geq \lt(2^{\frac{1-s}{2}}\,\frac{\Gamma\lt(
\frac{n+1-s}{4}\rt)}{\Gamma\lt(\frac{n-1+s}{4}\rt)}\rt)^2\,\int_{\R^n} \frac{f(x)^2}{|x|^{1-s}} dx,\quad f\in \dHs.
\end{equation}
This is a subclass of the so-called Pitt's inequality (or fractional Hardy inequality) which was first proved by Herbst \cite{H77}, based on dilation analytic techniques, and thereafter by Beckner \cite{B95}, based on the Stein-Weiss potential and Young's inequality, (see also \cite{Eil01,Ya99} for the other proofs). We remark that the Pitt's inequality holds in $\dHa$ for all $\al\in (0,n)$. Since $(1-s)/{2}\in (0,1)$ for $s\in (-1,1)$, it is then worthwhile to note that the proof of Theorem \ref{maintheorem} gives an alternative proof of the Pitt's inequality for $\al\in (0,1)$ (another proof for these $\al$ can be found in \cite[Proposition $4.1$]{FLS08} by using a ground state representation).

For any function $f\in \dHa$ with $\al\in (0,n/2)$, the sharp fractional Sobolev inequality (see \cite{CT04,Lieb83}) says that 
\begin{equation}\label{eq:fSinequality}
\lt(\int_{\R^n} |f(x)|^{\frac{2n}{n-2\al}} dx\rt)^{\frac{n-2\al}n} \leq \frac{\Gam\left(\frac{n-2\al}2\right)}{2^{2\al}\pi^\al\,\Gam\left(\frac{n+2\al}2\right)}\lt(\frac{\Gam(n)}{\Gam\left(\frac n2 \right)}\rt)^{\frac{2\al}n}\|f\|^2_{\dot{H}^\al(\R^n)}.
\end{equation}
Combining  \eqref{eq:fSinequality} and \eqref{eq:traceonboundary} derives a weighted trace Sobolev inequality which reads as follows
\begin{equation}\label{eq:traceSobolev}
\lt(\int_{\sp} |u(x,0)|^{2(s)^*}dx\rt)^{\frac{2}{2(s)^*}}\leq C_{n,s} \int_{\hs}|\na u|^2 t^{s} dx dt,
\end{equation}
where $2(s)^* =2n/(n-1+s)$, $s\in (-1,1)$, and the sharp constant $C_{n,s}$ is given by 
$$C_{n,s} =\frac1{2\pi^{\frac{1-s}2}} \frac{\Gam\lt(\frac{1-s}2\rt)\Gam\lt(\frac{n-1+s}2\rt)}{\Gam\lt(\frac{1+s}2\rt)\Gam\lt(\frac{n+1-s}2\rt)} \lt(\frac{\Gam(n)}{\Gam\lt(\frac n2\rt)}\rt)^{\frac{1-s}n}.$$
Equality occurs in \eqref{eq:traceSobolev}  if and only if $u$ is solution of \eqref{eq:extensionproblem} with the initial condition of the form $c(|x-x_0|^2 + t_0^2)^{-(n-1+s)/2}$ for some $c\in \R$, $x_0\in \R^n$, and $t_0 >0$.

The fractional Laplacian defined on subsets $\Om$ of $\R^n$ recently appears in \cite{CC10,CT10,T11}. Its extension problem  is to consider test functions in $\COR$. When $\Om$ is the half space $\Rn = \{x_n > 0\}$, we have the following results

\begin{theorem}\label{maintheorem2}
Let $s\in(-1,1)$, $n\geq 2$, and $\be\in[0,1]$. There exists a constant $k(s,\be)$ such that for any function $u\in C_0^\infty(\R^n_+\times\R)$, it holds

\begin{equation}\label{eq:mainresult2}
\i0i\int_{\Rn} |\na u|^2 t^s dx dt \geq \frac{\be^2}4 \i0i\int_{\Rn} \frac{u(x,t)^2}{x_n^2 } t^s dxdt +k(s,\be)\int_{\R^n_+} \frac{u(x,0)^2}{x_n^{1-s}} dx,
\end{equation}
where $k(s,\be)$ is given by
\begin{equation}\label{eq:ksbeta}
k(s,\be) = 2\frac{\Gam\lt(\frac{1+s}{2}\rt)}{\Gam\lt(\frac{1-s}{2}\rt)}\lt(\frac{\Gam\lt(\frac{3-s+\sqrt{1-\be^2}}{4}\rt)}{\Gam\lt(\frac{1+s +\sqrt{1-\be^2}}{4}\rt)}\rt)^2.
\end{equation}
The constant $k(s,\be)$ is optimal. Moreover, there is a positive constant $c>0$ such that the following inequality holds for all $u\in C_0^\infty(\ORn)$ 
\begin{align}\label{eq:improvementversion2}
\i0i\int_{\Rn} |\na u|^2 t^s dx dt &\geq \frac{\be^2}4 \i0i\int_{\Rn} \frac{u(x,t)^2}{x_n^2} t^s dxdt +k(s,\be)\int_{\R^n_+} \frac{u(x,0)^2}{x_n^{1-s}} dx\notag\\
&\quad\quad + c\lt(\int_{\R^n_+} |u(x,0)|^{\frac{2n}{n-1+s}} dx\rt)^{\frac{n-1+s}n}.
\end{align}
\end{theorem}
When $\Om$ is an arbitrary domain of $\R^n$, we have the following results under a special geometric assumption of the domain, 
\begin{theorem}\label{maintheorem3}
Let $s\in(-1,0]$, $n\geq 2$, $\be\in[0,1]$, and let $\Om$ be a proper domain of $\R^n$. Assume, in addition, that
\begin{equation}\label{eq:meanconvex}
-\De d(x) \geq 0, \quad x\in \Om,
\end{equation}
where $d(x) = \inf\{|x-y|\,:\, y\in \pa \Om\}$ is the distance from $x$ to $\pa \Om$, then there exists a constant $k(s,\be)$ such that for any function $u\in C_0^\infty(\Om\times\R)$, it holds

\begin{equation}\label{eq:mainresult3}
\int_0^\infty \int_{\Om} |\na u|^2 t^s dx dt \geq \frac{\be^2}4 \int_0^\infty\int_{\Om} \frac{u(x,t)^2}{d(x)^2 } t^s dxdt +k(s,\be)\int_\Om \frac{u(x,0)^2}{d(x)^{1-s}} dx,
\end{equation}
where $k(s,\be)$ is given by \eqref{eq:ksbeta}.

If there exists $x_0\in \pa \Om$ and $r > 0$ such that the part of the boundary $\pa \Om \cap B(x_0,r)$ is $C^1-$regular, then the constant $k(s,\be)$ is optimal.

Moreover, if $\Om$ is uniformly Lipschitz domain and has finite inner radius (that is, $\Rin(\Om):=\sup_{x\in \Om} d(x) < \infty$) and $s\in (-1,0)$ then there is a positive constant $c>0$ such that the following inequality holds for all $u\in C_0^\infty(\Om\times \R)$
 
\begin{align}\label{eq:improvementversion3}
\int_0^\infty\int_{\Om} |\na u|^2 t^s dx dt &\geq \frac{\be^2}4 \int_0^\infty\int_{\Om} \frac{u(x,t)^2}{d(x)^2 } t^s dxdt +k(s,\be)\int_\Om \frac{u(x,0)^2}{d(x)^{1-s}} dx\notag\\
&\quad\quad + c\lt(\int_\Om |u(x,0)|^{\frac{2n}{(n-1+s)}} dx\rt)^{\frac{n-1+s}n}.
\end{align}
\end{theorem}

The case $\beta =0$ in Theorems \ref{maintheorem2} and \ref{maintheorem3} are exact Theorem $1$ and $2$ in \cite{FMT13}, respectively. Our proofs below of Theorems \ref{maintheorem2} and \ref{maintheorem3} follow closely the ideas in the proof of Theorems $1$ and $2$ in \cite{FMT13}. We note that Theorem \ref{maintheorem3} is only stated for $s\in (-1,0]$. This condition is imposed to treat the term concerning to the quantity $-\De d$ (which is $0$ in the half space case). 

Both of Sobolev trace inequality and Hardy inequality have many applications, especially, to the boundary value problem for partial differential equation and nonlinear analysis. They have been developed by many authors in may different setting by many different methods (see~\emph{e.g.} \cite{Ad02,AVV10,AVV11,BM97,BV97,CR09,Da99,FS081,FMT07,FMT13,FS08,GS08,GGM04,GM08,VZ00}).

We conclude this introduction by introducing the logarithmic Hardy trace inequalities and logarithmic Sobolev trace inequalities which are the consequences of the weighted trace Hardy inequality \eqref{eq:traceHardy}, the weighted trace Sobolev inequality \eqref{eq:traceSobolev} and H\"older inequality. More precisely, we will prove the following theorem,
\begin{theorem}\label{logHS}
There exist positive constants $C_1, C_2\leq C_{n,s}$ such that
\begin{description}
\item (i) If $u\in \dot{W}(d_{\hs}^s,\hs)$, $\int_{\sp}u(x,0)^2dx =1$, there holds
\begin{equation}\label{eq:logSobtrace}
\int_{\sp} u(x,0)^2 \ln(u(x,0)^2) dx \leq \frac{n}{1-s} \ln\lt(C_1\int_{\hs}|\na u(x,t)|^2 t^{s} dx dt\rt).
\end{equation}
\item (ii) If $u\in \dot{W}(d_{\hs}^s,\hs)$, $\int_{\sp} \frac{u(x,0)^2}{|x|^{1-s}} dx =1$, we then have
\begin{equation}\label{eq:logHardytrace}
\int_{\sp} \frac{u(x,0)^2}{|x|^{1-s}} \ln\lt(\frac{u(x,0)^2}{|x|^{1-s-n}}\rt) dx \leq \frac n{1-s} \ln\lt(C_2\int_{\hs}|\na u(x,t)|^2 t^s dx dt\rt).
\end{equation}
\end{description}
\end{theorem}
The inequalities \eqref{eq:logSobtrace} and \eqref{eq:logHardytrace} are the trace versions of the logarithimic Sobolev inequality and logarithmic Hardy inequality obtained in \cite{DD02,DDFT10} by Dolbeault et al.. It should mention here that the inequalities \eqref{eq:traceSobolev} and \eqref{eq:logSobtrace} with the different constants were obtained by Xiao in \cite{X06}. In his paper, Xiao proved these inequalities only for the harmonic extension of the functions from $\dHs$.

Denoting $C_{LS}(n,s)$ and $C_{LH}(n,s)$ the best constants for which the inequalities\eqref{eq:logSobtrace} and \eqref{eq:logHardytrace} hold, respectively. From Theorem \ref{logHS}, we see that $C_{LS}(n,s)$ and $C_{LH}(n,s)$ are dominated by $C_{n,s}$. However, we do not know the explict values of $C_{LS}(n,s)$ and $C_{LH}(n,s)$ in general. If we restrict \eqref{eq:logSobtrace} and \eqref{eq:logHardytrace} to the radial functions in $\dot{W}(d_{\hs}^s,\hs)$, we will obtain the following results. Denoting $C_{LS,r}(n,s)$ and $C_{LH,r}(n,s)$ the best constants for which the inequalities \eqref{eq:logSobtrace} and \eqref{eq:logHardytrace} hold for any radial function in $\dot{W}(d_{\hs}^s,\hs)$, respectively, then
\begin{theorem}\label{Radialcase}
We have
\begin{equation}\label{eq:CLSr}
C_{LS,r}(n,s) =\frac{8}{n(1-s)e}\frac{\Gam\lt(\frac{n+1+s}{2}\rt)}{\Gam\lt(\frac n2\rt)\Gam\lt(\frac{1+s}{2}\rt)}\lt(\frac{(1-s)\Gam\lt(\frac n2\rt)}{2 \pi^{\frac n2} \Gam\lt(\frac n{1-s}\rt)}\rt)^{\frac{1-s}{n}},
\end{equation}
and 
\begin{equation}\label{eq:CHSr}
C_{LH,r}(n,s) =\frac{2\Gam\lt(\frac{n+1+s}{2}\rt)}{\pi^{\frac{1-s}{2}} \Gam\lt(\frac n2 +1\rt)\Gam\lt(\frac{1+s}{2}\rt)} \lt(\frac{2n-1+s}{(n-1+s)^2}\rt)^{1-\frac{1-s}{2n}} \lt(\frac{(1-s)\Gam\lt(\frac n2\rt)^2}{8\pi e}\rt)^{\frac{1-s}{2n}}.
\end{equation}
\end{theorem}

The rest of this paper is organized as follows. In the next section, we collect some useful properties of the hypergeometric functions which are used intensively in this paper. In section \S3, we prove Theorem \ref{maintheorem}. Section \S4 is devoted to the proof of Theorem \ref{maintheorem2} and Theorem \ref{maintheorem3} and derive some their consequences. In the last section, we prove the logarithmic Sobolev trace inequalities and the logarithmic Hardy trace inequalities presented in Theorem \ref{logHS}, and compute the constants $C_{LS,r}(n,s)$ and $C_{LH,r}(n,s)$.


\section{Preliminaries}
In this section, we collect some main properties of hypergeometric functions that are extensively used throughout the next sections. We refer the readers to the books \cite[Section $15$]{AS64}, \cite[Section $2$]{PZ03} for more details about these functions. 

Let $a,b,c$ be complex numbers. Considering the hypergeometric differential equation
\begin{equation}\label{eq:hyperequation}
z(1-z) \omega''(z) + [c-(a+b+1)z] \omega'(z) -ab \omega(z) =0.
\end{equation}
If $c$ is not an integer, then the general solution of \eqref{eq:hyperequation} is given by 
\begin{equation}\label{eq:hypersolution}
\omega(z) = C_1 F(a,b,c,z) + C_2 z^{1-c} F(a-c+1, b-c+1, 2-c,z),
\end{equation}
for some complex constants $C_1,C_2$ (see \cite[page $562$]{AS64} or \cite[page $257$]{PZ03}). Here the hypergeometric function $F(a,b,c,z)$ is defined by the Gauss series (see \cite[page $556$]{AS64})
\begin{equation}\label{eq:hyperfunction}
F(a,b,c,z) = \sum_{k=0}^\infty \frac{(a)_k(b)_k}{(c)_k} \frac{z^k}{k!},
\end{equation}
on the disk $|z|< 1$ and by analytic continuation on whole complex plane cut along the interval $[1,\infty)$. We also use the notations $(a)_k = a (a+1)\cdots (a+k)$ and $(a)_0 =1$ for convenience. 

Note that the hypergeometric serie \eqref{eq:hyperfunction} is absolutely convergent if $|z| < 1$. The convergence also extends over the circle $|z| =1$ if $c-a-b >0$, while the serie converges at all points of the circle except the point $z =1$ if $c-a-b =0$. More precisely, we have the following asymptotic behavior of $F(a,b,c,z)$ when $z$ tends to $1$ (see \cite{AVF12}):
\begin{equation}\label{eq:lonhon}
 F(a,b,c,1) =\frac{\Gam(c)\Gam(c-b-a)}{\Gam(c-a)\Gam(c-b)}, \quad \text{if } c > a+b
\end{equation}
\begin{equation}\label{eq:bang}
\lim\limits_{z\to 1}\frac{F(a,b,c,z)}{\ln(1-z)} = -\frac{\Gam(a+b)}{\Gam(a)\Gam(b)},\quad \text{if } c = a+b
\end{equation}
\begin{equation}\label{eq:nhohon}
\lim\limits_{z\to 1}\frac{F(a,b,c,z)}{(1-z)^{c-a-b}} = \frac{\Gam(c)\Gam(a+b-c)}{\Gam(a)\Gam(b)},\quad\text{if } c < a+b.
\end{equation}

In the sequel, we quote some special expression formulas for hypergeometric function $F(a,b,c,z)$ that are useful for our purpose:
\begin{description}
\item (i) If none of $a,b,c, c-a, c-b, a-b,$ and $b-a$ is equal to a nonpositive integer, then we have (see \cite[$15.3.7$]{AS64})
\begin{align}\label{eq:nghichdao}
F(a,b,c,z) &= \frac{\Gam(c) \Gam(b-a)}{\Gam(b) \Gam(c-a)} (-z)^{-a} F\lt(a, 1-c+a,1-b+a,\frac1z\rt)\notag\\
&\quad +\frac{\Gam(c) \Gam(a-b)}{\Gam(a) \Gam(c-b)} (-z)^{-b} F\lt(b, 1-c+b,1-a+b,\frac1z\rt),
\end{align}
when $|\arg(-z)| < \pi$.
\item (ii) If $c =a+b$, then we have (see \cite[$15.3.10$]{AS64})
\begin{align}\label{eq:mbang0}
F(a,b,a+b,z)=\frac{\Gamma(a+b)}{\Gamma(a)\Gamma(b)} \sum_{k=0}^\infty & \frac{(a)_k(b)_k}{k!}[2\psi(k+1) -\psi(a+k)\notag\\
& -\psi(b+k) -\ln(1-z)](1-z)^k,
\end{align}
when $|\arg(1-z)| < \pi$ and $|1-z| < 1$. Here $\psi$ denotes the logarithmic derivative of Gamma function, i.e, $\psi(z) = \frac{\Gamma'(z)}{\Gamma(z)}.$
\item (iii) If $a\notin \Z$ then we have (see \cite[$15.3.13$]{AS64})
\begin{align}\label{eq:logarith}
F(a,a,2a,z) = \frac{\Gam(2a)}{\Gam(a)^2 (-z)^a} \sum_{k=0}^\infty &\frac{(a)_k(1-a)_k}{(k!)^2} \frac1{z^k} [\ln(-z) + 2\psi(k+1)\notag\\
 &-\psi(a+k) - \psi(a-k)]
\end{align}
when $|\arg(-z)| < \pi, |z| > 1$. 
\item (iv) The hypergeometric function satisfies the following differential formulas (see for instance \cite[$15.2.1$, $15.2.4$, $15.2.6$, $15.2.9$]{AS64})
\begin{equation}\label{eq:daohamhyper}
\frac{d}{dz}F(a,b,c,z) = \frac{ab}c F(a+1,b+1,c+1,z),
\end{equation}
\begin{equation}\label{eq:daohamhyper1}
\frac{d^n}{dz^n}[z^{c-1} F(a,b,c,z)] = (c-n)_n z^{c-n-1}F(a,b,c-n,z),
\end{equation}
\begin{equation}\label{eq:daohamhypercapn}
\frac{d^n}{dz^n}[(1-z)^{a+b-c}F(a,bc,z)] = \frac{(c-a)_n(c-b)_n}{(c)_n} (1-z)^{a+b-c-n} F(a,b,c+n,z),
\end{equation} 
\begin{equation}\label{eq:daohamhypercapn2}
\frac{d^n}{dz^n}[z^{c-1}(1-z)^{a+b-c}F(a,b,c,z)] =(c-n)_n z^{c-n-1}(1-z)^{a+b-c-n} F(a-n,b-n,c-n,z).
\end{equation}
\end{description}

Let us conclude this section by the following useful lemma.
\begin{lemma}\label{fundamental2}
Let $a,b > 0$, $c\in (0,1)$ be such that $a+b \geq c$, setting
$$C =-\frac{\Gamma(c) \Gamma(a+1-c) \Gamma(b+1-c)}{\Gamma(2-c) \Gamma(a)\Gamma(b)},$$
and the function
$$\eta(z) = F(a,b,c,z) + C z^{1-c} F(a+1-c,b+1-c,2-c,z).$$
Then there exists the limit $\lim\limits_{z\to 1} \eta(z)$.
\end{lemma}
\begin{proof}
We divide our proof into two cases:

\emph{Case $1$}: $a+b = c+m$ for some integer $m =0,1,2,\ldots$. We argue inductively in $m$. If $m =0$ the conclusion follows from \eqref{eq:lonhon} and the choice of $C$. Suppose that the conclusion holds for any $a+b =c +k$, $0\leq k \leq m$. We will show that it also holds for any $a+b =c+m+1$. Indeed, by the choice of $C$ and \eqref{eq:nhohon}, we then have $\lim_{z\to 1} (1-z)^{m+1} \eta(z) = 0$. Making the uses of L'H\^{o}pital theorem and the differential formulas \eqref{eq:daohamhypercapn}, \eqref{eq:daohamhypercapn2}, we obtain
\[
\lim_{z\to 1} \eta(z) = \lim_{z\to 1} \frac{[(1-z)^{m+1}\eta(z)]'}{[(1-z)^{m+1}]'} =-\frac{(c-a)(c-b)}{(m+1)c}\lim_{z\to 1} \eta_1(z),
\]
where
\[
\eta_1(z) = F(a,b,1+c,z) +\frac{c(1-c)C}{(c-a)(c-b)}z^{-c} F(a-c,b-c,1-c,z)
\]
Since $a+b = c+1 +m$, the choice of $C$ and our induction assumption, there exists the limit $\lim_{z\to 1}\eta_1(z)$. So does the limit $\lim_{z\to 1} \eta(z)$.

\emph{Case $2$:} $a+b = c+ m + \al$ for $\alpha \in (0,1)$ and for some integer $m =0,1,2,\ldots$. We also argue inductively in $m$. If $m =0$, it is implied from the choice of $C$ and \eqref{eq:nhohon} that $\lim_{z\to 1} (1-z)^{\alpha} \eta(z) = 0$. Using L'H\^{o}pital theorem, the differential formulas \eqref{eq:daohamhypercapn}, \eqref{eq:daohamhypercapn2} and \eqref{eq:lonhon} implies the conclusion when $m=0$. The rest of argument is completely similar with the one of \emph{Case $1$}.
\end{proof}

\section{Proof of Theorem \ref{maintheorem}}
We follow the ideas in \cite{AVF12,T15}. To do this, let us define an energy functional on $\dWds$ by
\begin{equation}\label{eq:energyfunctional}
J[u] = \int_{\mC} |\na u(x)|^2 d_\mC(x)^{s} dx - \frac{(\be -2)^2}{4} \int_{\mC} \frac{u(x)^2}{|x|^2} d_\mC(x)^{s} dx.
\end{equation}
The Euleur-Lagarange equation of the functional $J$ is given by
\begin{equation}\label{eq:ELequation}
\De u(x) +s\frac{\la \na u(x),\na d_\mC(x)\ra}{d_\mC(x)} +\frac{(\be-2)^2}{4}\frac{u(x)}{|x|^2} =0 \quad\mbox{in}\quad \mC.
\end{equation}
We next construct a positive solution $\vphi$ of \eqref{eq:ELequation} with the condition $\vphi(x) = |x|^{-(n_s-2)/2}$ for $x\in\pa\mC$. Writing $\vphi$ in the form
\begin{equation}\label{eq:changevariable}
\vphi(x) = |x|^{-\frac{n_s-2}2} \om\lt(\frac{d_\mC(x)^2}{|x|^2}\rt).
\end{equation}
Then $\om$ is solution of the equation
\begin{equation}\label{eq:equivalentversion}
\begin{cases}
z(z-1)\om''(z) +\lt(\frac{n_s}2 z-\frac{1+s}{2}\rt)\om'(z) +\lt[\frac{(n_s-2)^2}{16} -\frac{(\be-2)^2}{16}\rt]\om(z) =0, &\mbox{$z\in(0,1)$}\\
\om(0) =1,\quad \exists\lim\limits_{z\to 1} \om(z) \in \R.
\end{cases}
\end{equation}
Lemma \ref{fundamental2} implies that the function $\om$ must have the form
\begin{align}\label{eq:omega}
\om(z) &= F\lt(\frac{n_s+\be-4}{4}, \frac{n_s-\be}{4},\frac{1+s}{2}, z\rt)\notag \\
&\quad -\frac{1}{1-s} H(n,s,\be) z^{\frac{1-s}{2}} F\lt(\frac{n_s+\be}{4} -\frac{1+s}{2}, \frac{n_s-\be}{4}+\frac{1-s}{2},\frac{3-s}{2}, z\rt).
\end{align}

Hence we have shown that
\begin{proposition}\label{constructofsolution}
Let $s\in (-1,1)$, $2\leq \be < n_s$ and let $H(n,s,\be)$ be given by \eqref{eq:constantH}. Then the function
\begin{align}\label{eq:solution}
\vphi(x)& = \frac{F\lt(\frac{n_s+\be-4}4,\frac{n_s-\be}4, \frac{1+s}2,\frac{d_\mC(x)^2}{|x|^2}\rt)}{(|x|^2 +t^2)^{\frac{n_s-2}4}}\notag\\
&\quad -\frac{H(n,s,\be)d_\mC(x)^{1-s}}{(1-s) (|x|^2)^{\frac{n_s-2s}4}}F\lt(\frac{n_s+\be}4-\frac{1+s}{2},\frac{n_s-\be}4+\frac{1-s}{2}, \frac{3-s}{2},\frac{d_\mC(x)^2}{|x|^2}\rt)
\end{align}
solves the equation \eqref{eq:ELequation} with condition $\vphi(x) = |x|^{-\frac{n_s-2}2}$ on $\pa \mC$.
\end{proposition}

Note that the solution $\vphi$ above does not belong to $\dWds$. Let us collect some useful properties of the function $\om$ defined by \eqref{eq:omega}.

\begin{proposition}\label{phiduong}
Let $s\in (-1,1)$, and let $\om$ be given by \eqref{eq:omega}, then 
\begin{description}
\item (i) $\om(z) > 0$ for any $z\in [0,1]$.
\item (ii) The following equality holds true
\begin{equation}\label{eq:limom}
\lim\limits_{z\to 0^+} z^s\, \frac{d (\om(z^2))}{dz} = -H(n,s,\be).
\end{equation}
\item (iii) $\om'(z) < 0$ for any $z\in [0,1]$ and there exists constant $C>0$ such that
\begin{equation}\label{eq:om'est}
|\om'(z)| \leq C(1+ z^{-\frac{1+s}2}), \quad \forall\, z\in [0,1].
\end{equation}
\end{description}
\end{proposition}
From \eqref{eq:limom} we easily deduce that 
\begin{equation}\label{eq:limphi}
\lim\limits_{d_\mC(x)\to 0}\frac{d_\mC(x)^{s}\la\na \vphi(x),\na d_\mC(x)\ra}{|x|^{s-1}\vphi(x)} = -H(n,s,\be).
\end{equation}
\begin{proof}
Let us first prove part $(i)$. If $n=2$ then the conclusion immediately follows from the definition of the hypergeometric functions, $s\in (-1,1)$ and the simple inequality
$$\frac{\Gam(a)}{\Gam(b)} \geq \frac{\Gam(a+\al)}{\Gam(b+\al)},\quad b > a >0,\, \forall\, \al \geq 0,$$
which is a consequence of the convexity of the function $t\mapsto   \ln (\Gam(t))$. 

The case $n\geq 3$ is more complicated and is proved by using the maximum principle. To do this, let us define an auxiliary function
$$v(z)= (1-z)^{\ga} \om(z),$$
for some $\ga >0$ which will be chosen later. It is obvious that $v(0) =\om(0) =1$ and $v$ solves the equation
$$-z(1-z)^2 v''(z) + \lt(\lt(\frac{n_s}2 -2\ga\rt)z -\frac{1+s}2\rt)(1-z) v'(z) + f(z) v(z)=0,$$
with
$$f(z) = \lt(\frac{(n_s-2)^2}{16} -\frac{(\be-2)^2}{16} -\ga\frac{1+s}2\rt)(1-z) +\ga\lt(\frac n2 -1 -\ga\rt) z.$$
Because of both of the hypergeometric functions appearing in \eqref{eq:omega} satisfy the condition $c-a-b = (2-n)/2 \leq 0$, then by \eqref{eq:lonhon} and \eqref{eq:bang}, we have $v(1) =0$.

Since $\be < n_s$ and $n\geq 3$, we can choose $\ga$ such that
$$0 < \ga < \min\lt\{\frac{2}{1+s}\lt(\frac{(n_s-2)^2}{16} -\frac{(\be-2)^2}{16}\rt), \frac n2 -1\rt\},$$
which then implies $f(z) > 0$ for any $z\in (0,1)$. The function $v$ attains its minimum at a point $z_0\in [0,1]$. Since $v(0) =1 >v(1) =0$, we must have $z_0 > 0$. If $z_0 =1$ then $v(z) \geq v(1) =0$ for any $z\in [0,1]$. Otherwise, we have $v'(z_0) =0$, $v''(z_0) \geq 0$ and hence
$$v(z) \geq v(z_0) = \frac{z_0(1-z_0)^2 v''(z_0)}{f(z_0)} \geq 0,\quad\forall\, z\in [0,1].$$
We have shown that $v\geq 0$ on $[0,1]$, hence so is $\om$.

A standard maximum principle (see, e.g, Theorem $3$ in \cite[page $6$]{PW}) shows that $\om(z)> 0$ for any $z\in (0,1)$.

It remains to verify that $\om(1)>0$. We argue by contradiction. Suppose that $\om(1) =0$. It follows from the differential formulas \eqref{eq:daohamhyper} and \eqref{eq:daohamhyper1} that
\begin{align}\label{eq:daoham}
\om'(z)& = \frac{\lt(\frac{n_s+\be}4-1\rt)\frac{n_s-\be}4}{\frac{1+s}2}F\lt(\frac{n_s+\be}4, \frac{n_s-\be}4+1,\frac{3+s}2,z\rt)\notag\\
&\quad + C\frac{1-s}2 z^{-\frac{1+s}2}F\lt(\frac{n_s+\be}4-\frac{1+s}2, \frac{n_s-\be}4+\frac{1-s}2,\frac{1-s}2,z\rt)
\end{align} 
and
\begin{align*}
\om''(z)& = \frac{\frac{n_s+\be}4\lt(\frac{n_s+\be}4-1\rt)\frac{n_s-\be}4\lt(\frac{n_s-\be}4+1\rt)}{\frac{1+s}2 \frac{3+s}2}F\lt(\frac{n_s+\be}4 +1, \frac{n_s-\be}4+2,\frac{5+s}2,z\rt)\notag\\
&\quad - C\frac{1-s}2 \frac{1+s}2 z^{-\frac{3+s}2}F\lt(\frac{n_s+\be}4-\frac{1+s}2, \frac{n_s-\be}4+\frac{1-s}2,-\frac{1+s}2,z\rt)
\end{align*}
Because of $C = -H(n,s,\be)/(1-s)$ and Lemma \ref{fundamental2}, we see that the limits $\lim\limits_{z\to 1} \om'(z)$ and $\lim\limits_{z\to 1} \om''(z)$ exist. Hence $\om \in C^1((0,1])$ and by \eqref{eq:equivalentversion}, we get
\begin{equation}\label{eq:danhgiadaoham}
\om'(1) = -\frac 2n \lt(\frac{(n_s-2)^2}{16} -\frac{(\be-2)^2}{16}\rt) \om(1)= 0.
\end{equation}
Fix a $\de \in (0,1)$ such that $\de n_s -1-s >0$. Let us define the new auxiliary function $u(z) = e^{-\alpha(z-1)} -1$ with $\alpha >0$ will be chosen below and the second differential operator
$$L = z(z-1) \frac{d^2}{dz^2} + \lt(\frac{n_s}{2} -\frac{1+s}{2}\rt)\frac{d}{dz}+ \lt(\frac{(n_s-2)^2}{16} -\frac{(\be-2)^2}{16}\rt).$$
An easy computation shows that
$$e^{\alpha(z-1)}Lu(z) = z(z-1)\alpha^2 -\lt(\frac{n_s}{2}z -\frac{1+s}{2}\rt) \alpha +\lt(\frac{(n_s-2)^2}{16} -\frac{(\be-2)^2}{16}\rt)(1 -e^{\alpha(z-1)}).$$
Therefore, by choosing $\alpha >0$ large enough, $Lu(z) < 0$ for all $z\in [\de,1]$. Since $\om(\de) >0$ and $u(\de) >0$, we can choose $0 < \ep < \om(\de)/u(\de)$. Let us define the function
$$\om_\ep(z) = \om(z) -\ep u(z).$$
We then have $L\om_\ep(z) >0$ on $[\delta,1)$ and $\om_\ep(\delta) >\om_\ep(1) =0$. An application of maximum principle yields that $\om_\ep$ has to attain its minimun on $[\de,1]$ at $z =1$. Therefore, we have $\om_\ep(z) \geq 0$ for any $z\in [\de,1]$. As a consequence, we have $\om_\ep'(1) \leq 0$, or equivalently
$$\om'(1) \leq -\ga \ep < 0$$ 
which contradicts with \eqref{eq:danhgiadaoham}. Therefore $\om(1) = 0$ cannot occur, or $\om(1) >0$.

The part $(ii)$ is immediately implied from \eqref{eq:daoham}.

By part $(i)$ we have
$$\om'(1) = -\frac 2n \lt(\frac{(n_s-2)^2}{16} -\frac{(\be-2)^2}{16}\rt) \om(1) <0.$$
It yields from \eqref{eq:daoham} that $\lim\limits_{z\to 0} \om'(z) < 0$. A standard maximum principle argument shows that $\om'(z) < 0$ for all $z$. The estimate \eqref{eq:om'est} is derived from \eqref{eq:daoham} and the fact $|\om'(1)| < \infty$.
\end{proof}

{\bf Proof of Theorem \ref{maintheorem}:} With the function $\vphi$ on hand, the proof of Theorem \ref{maintheorem} is completely similar with the proof of Theorems $1$ and $2$ in \cite{T15} by factorizing the function $u$ in the form $u =v \vphi$ (note that $\vphi$ is strict positive in $\mC$), by using divergence theorem and the limit \eqref{eq:limphi}. We refer the reader to the paper \cite{T15} for more details on the proof.

The optimality of constant $H(n,s,\beta)$ is verified by truncating the function $\vphi$ as follow. Taking a function $\phi\in C_0^\infty(\R^{n+1})$ such that $\phi(x,t) =1$ if $|(x,t)| \leq 1$, and $\phi(x,t) =0$ if $|(x,t)| \geq 2$. For $\ep > 0$, denote $\phi_\ep((x,t)) = \phi\lt((x,t)/\ep\rt)$, and $\vphi_\ep = \vphi(1-\phi_\ep) \phi_{\frac1\ep}.$ We can readily check that the functions $\vphi_\ep$ give us the optimality of $H(n,s,\beta)$ by letting $\ep$ tend to $0$.

We conclude this section by showing that inequality \eqref{eq:maininequality} can be improved by adding an extra positive term on the right hand side as done in \cite[Theorem 2]{T15} for the upper half space $\hs$. To do this, let us introduce the recursively defined functions, for $|x|\leq 1$,
\[
X_1(|x|) = \frac1{1-\ln |x|},\quad X_k(|x|) = X_1(X_{k-1}(|x|)),\quad k= 2,3,\ldots,
\]
and
\[
P_k(|x|) =X_1(|x|)X_2(|x|)\cdots X_k(|x|),\quad k= 1,2,\ldots.
\]
Let $U$ be a generic bounded domain of $\R^{n+1}$ such that the origin is in the interior of $U$. Denote $D = \sup_{x\in U\cap \mC} |x|$. Then the following inequality holds for any function $u\in C_0^\infty(U)$
\begin{align}\label{eq:improvedH}
\int_{U\cap\mC} |\na u(x)|^2 d_\mC(x)^2 dx &\geq \frac{(\beta-2)^2}4 \int_{\mC} \frac{u(x)^2}{|x|^2} d_\mC(x)^s dx + H(n,s,\be)\int_{\pa\mC}\frac{u(x)^2}{|x|^{1-s}} d\mH^n(x)\notag\\
&\quad + \frac14 \sum_{i=1}^\infty \int_{U\cap \mC} P_i(|x|/D)^2 \frac{u(x)^2}{|x|^2} d_\mC(x)^s dx.
\end{align} 
The proof of \eqref{eq:improvedH} is similar with the one of Theorem $2$ in \cite{T15}, hence we drop it here. 

\section{Proof of Theorem \ref{maintheorem2} and Theorem \ref{maintheorem3}} 
Let us denote
$$J[u] = \int_0^\infty\int_{\R_+^n} |\na u|^2 t^s dx dt -\frac{\be^2}4 \i0i \int_{\Rn} \frac{u(x,t)^2}{x_n^2} t^s dx dt,$$
the energy functional associated to the inequality \eqref{eq:mainresult2}. Its Euler-Lagrange equation is given by
\begin{equation}\label{eq:general}
\De_{(x,t)} u(x,t) + s\frac{\pa_t u(x,t)}t + \frac{\be^2}4 \frac{u(x,t)^2}{x_n^2} =0, \quad (x,t) \in \Rn\times (0,\infty).
\end{equation}
For convenience, we use the notation $L_s = \De_{(x,t)} + s\pa_t/t $.

We next construct a positive solution $\phi$ for the equation \eqref{eq:general}. Writing $\phi$ in form $\phi(x,t) = x_n^{- s/2} \om\lt({t}/{x_n}\rt)$, then $\om$ solves the equation
\begin{equation}\label{eq:EL2}
(y^2 +1) \om''(y) + \lt((s+2)y + \frac sy\rt)\om'(y) + \frac{s(s+2) + \be^2}4 \om(y) =0,\quad y>0.
\end{equation} 

In the sequel, we use the following notation, for two function $f,g$ defined in a set $\Om$, $f\sim g$ if there exist two positive constants $C, c$ such that $c f \leq g\leq C f$ in $\Om$.

\begin{lemma}\label{solution2}
Given $s\in (-1,1)$, there exists a solution $\om$ of \eqref{eq:EL2} with the conditions 
$$\om(0) = 1, \quad \lim_{y\to \infty} \om(y) = 0.$$
This solution is postive and strictly decreasing. Moreover, we have
\begin{description}
\item (i) There is a positive constant $k(s,\be)$ such that 
$$\lim_{y\to\infty} y^s \om'(y) = -k(s,\be),$$
where
$$k(s,\be) =(1-s)\frac{\Gam\lt(\frac{1+s}{2}\rt)\Gam\lt(\frac{3-s+\sqrt{1-\be^2}}{4}\rt)^2}{\Gam\lt(\frac{3-s}{2}\rt)\Gam\lt(\frac{1+s +\sqrt{1-\be^2}}{4}\rt)^2}.$$
\item (ii) For any $y >0$,
$$\om(y) \sim (1+y^2)^{-\frac{1+s +\sqrt{1-\be^2}}4},\quad \om'(y)\sim -y^{-s}(1+y^2)^{-\frac{3-s+\sqrt{1-\be^2}}{4}},$$
and 
$$\lim_{y\to \infty} \frac{y \om'(y)}{\om(y)} = -\frac{1+s+\sqrt{1-\be^2}}{2}.$$
\item (iii) It holds true that
$$k(s,\be) = \int_0^\infty y^s(1+y^2) (\om'(y))^2 dy -\frac{s(s+2) + \be^2}4 \i0i \om(y)^2 y^s dy.$$
\item (iv) If $s\in (-1,0]$ then $y\om'(y) + s \om(y)/2 \leq 0$, and if $s\in (-1,0)$ then 
$$y\om'(y) + \frac s2 \om(y) \sim - \om(y), \quad y > 0.$$
\end{description}
\end{lemma}
\begin{proof}
We divide our proof into two cases.

\emph{Case $1$:} If $\be^2 +s(s+2) \not=0$. Making the change of function by $\om(y) = \eta(-y^2)$, then $\eta$ solves the equation
$$z(z-1) \eta''(z) +\lt(\frac{s+3}2 z -\frac{1+s}2\rt)\eta'(z) + \frac{s(s+2)+\be^2}{16} \eta(z) =0, \quad z\in(-\infty,0).$$
By \eqref{eq:hyperequation} and \eqref{eq:hypersolution}, the general solution $\eta$ is given by
\begin{align*}
\eta(z)& = C_1 F\lt(\frac{1+s +\sqrt{1-\be^2}}{4},\frac{1+s -\sqrt{1-\be^2}}{4},\frac{1+s}2, z\rt)\\
&\quad + C_2 z^{\frac{1-s}2}F\lt(\frac{3-s+\sqrt{1-\be^2}}{4}, \frac{3-s-\sqrt{1-\be^2}}{4}, \frac{3-s}{2}, z\rt). 
\end{align*}
Therefore, we obtain the form of $\om$ as follow
\begin{align*}
\om(y)&= C_1 F\lt(\frac{1+s +\sqrt{1-\be^2}}{4},\frac{1+s -\sqrt{1-\be^2}}{4},\frac{1+s}2, -y^2\rt)\\
&\quad + C_2 e^{i\frac{\pi}2(1-s)}y^{1-s}F\lt(\frac{3-s+\sqrt{1-\be^2}}{4}, \frac{3-s-\sqrt{1-\be^2}}{4}, \frac{3-s}{2}, -y^2\rt). 
\end{align*}

Taking $C_1 = 1$ we obtain $\om(0) =1$. It remains to choose $C_2$ such that there exists the limits $\lim_{y\to\infty} \om(y)$. Since both of the hypergeometric functions appearing in the solution $\om$ satisfy $b < a$ and $a+b =c$. If $\beta < 1$, by using the formula \eqref{eq:nghichdao}, we obtain
\begin{align*}
\om(y)& = y^{-\frac{1+s -\sqrt{1-\be^2}}2}\left(\frac{\Gam\lt(\frac{1+s}{2}\rt)\Gam\lt(\frac{\sqrt{1-\be^2}}{2}\rt)}{\Gam\lt(\frac{1+s +\sqrt{1-\be^2}}{4}\rt)^2} + C_2 e^{i\frac{\pi}2(1-s)} \frac{\Gam\lt(\frac{3-s}{2}\rt)\Gam\lt(\frac{\sqrt{1-\be^2}}{2}\rt)}{\Gam\lt(\frac{3-s+\sqrt{1-\be^2}}{4}\rt)^2}\rt)\\
&\qquad\qquad\qquad + O(y^{-\frac{1+s +\sqrt{1-\be^2}}2}).
\end{align*}
Choosing
$$C_2 = -e^{-i\frac{\pi}2(1-s)} \frac{\Gam\lt(\frac{1+s}{2}\rt)\Gam\lt(\frac{3-s+\sqrt{1-\be^2}}{4}\rt)^2}{\Gam\lt(\frac{3-s}{2}\rt)\Gam\lt(\frac{1+s +\sqrt{1-\be^2}}{4}\rt)^2}$$
ensures that $\lim_{y\to\infty} \om(y) =0$.

If $\be =1$, by using the formula \eqref{eq:logarith} we have that $\lim_{y\to\infty} \om(y) =0$ for any choice of constant $C$. So, in this case, we choose
$$C_2 =-e^{-i\frac{\pi}2(1-s)} \frac{\Gam\lt(\frac{1+s}{2}\rt)\Gam\lt(\frac{3-s}{4}\rt)^2}{\Gam\lt(\frac{3-s}{2}\rt)\Gam\lt(\frac{1+s}{4}\rt)^2}.$$ 
Hence the function $\om$ is completely determined. Using the formula \eqref{eq:daohamhyper}, we readily verify that
$$\lim_{y\to 0}\, y^s\om'(y) = (1-s) e^{-i\frac\pi 2(1-s)} C_2 = -(1-s)\frac{\Gam\lt(\frac{1+s}{2}\rt)\Gam\lt(\frac{3-s+\sqrt{1-\be^2}}{4}\rt)^2}{\Gam\lt(\frac{3-s}{2}\rt)\Gam\lt(\frac{1+s +\sqrt{1-\be^2}}{4}\rt)^2}.$$
The part $(i)$ then is proved. Since both of the hypergeometric functions appearing in the solution $\om$ satisfy $a < b+1$ if $\beta < 1$, then after some straightforward calculations we have
\begin{equation*}
\lim_{y\to \infty} y^{\frac{1+s +\sqrt{1-\be^2}}2}\om(y) = 
\begin{cases}
\frac{\Gam\lt(\frac{1+s}{2}\rt)\Gam\lt(-\frac{\sqrt{1-\be^2}}{2}\rt)}{\Gam\lt(\frac{1+s -\sqrt{1-\be^2}}{4}\rt)^2}\lt(1-\frac{\sin^2\lt(\frac{1+s +\sqrt{1-\be^2}}{4}\pi\rt)}{\sin^2\lt(\frac{1+s -\sqrt{1-\be^2}}{4}\pi\rt)}\rt)&\mbox{if $\be <1$,}\\
2\frac{\Gam\lt(\frac{1+s}2\rt)}{\Gam\lt(\frac{1+s}4\rt)^2} \lt(\psi\lt(\frac{3-s}4\rt) -\psi\lt(\frac{1+s}4\rt)\rt)&\mbox{if $\be =1$}
\end{cases}
\end{equation*}
and
$$\lim_{y\to \infty} y^{\frac{3+s +\sqrt{1-\be^2}}2}\om'(y) = -\frac{1+s+\sqrt{1-\be^2}}2\lim_{y\to \infty} y^{\frac{1+s +\sqrt{1-\be^2}}2}\om(y).$$
Hence
$$\lim_{y\to \infty} \frac{y\om'(y)}{\om(y)} = -\frac{1+s+\sqrt{1-\be^2}}2.$$ 

Writing the function $\om$ in the form 
$$\om(y) = (1+y^2)^{-\frac{1+s-\sqrt{1-\be^2}}4} \si\lt(\frac1y\rt)$$
then $\si$ solves the equation
$$(1+y^2)^2 \si''(y) + \lt((2-s)y + \frac{1-\sqrt{1-\be^2}}{y}\rt)(1+y^2) \si'(y) - \frac{(1+s-\sqrt{1-\be^2})^2}4 \si(y) =0,$$
with the conditions $\si(0) =0$ and $\si(\infty) = 1$. A standard maximum principal argument shows that $\si$ is positive on $(0,\infty)$, hence so is $\om$. Moreover, we have seen that $\lim\limits_{y\to \infty}y^{(3+s +\sqrt{1-\be^2})/2}\om'(y) <0$, and $\lim\limits_{y\to \infty} y^s \om'(y) < 0$. These inequalities together the positivity of $\om$ imply that $\om'(y) < 0$. Indeed, suppose that there exists $y$ such that $\om'(y) =0$. If $s(s+2) + \be^2 > 0$, let $y_0$ be the smallest of such $y$'s, then $y_0 > 0$, $\om'(y_0) =0$ and $\om'(y) < 0$ for any $y\in (0,y_0)$. From the equation for $\om$ we have $\om''(y_0) < 0$, hence $\om'(y) > 0$ with $y\in (y_0-\de,y_0)$ for some $\de> 0$ small enough. This is impossible. If $s(s+2) +\be^2 < 0$, we use the same argument with $y_1$ is the largest of $y$ such that $\om'(y) =0$.

The positivity and monotonicity of $\om$ together its asymptotic behaviors yield part $(ii)$ and part $(iv)$. The part $(iii)$ is derived from part $(ii)$ and integration by parts.

{\bf Case 2:} When $\be^2 +s(s+2) =0$, in this case we have
$$(y^s(1+y^2)\om'(y))' = 0.$$
Hence $\om'(y) = Cy^{-s}(1+y^2)^{-1}$ and
$$\om(y) = 1 + C\int_0^y t^{-s} (1+t^2)^{-1} dt,$$
 for some constant $C$. Choosing
 $$C = -\lt(\int_0^\infty t^{-s} (1+t^2)^{-1} dt\rt)^{-1} = -\frac{2}{\Gam\lt(\frac{1-s}2\rt)\Gam\lt(\frac{1+s}2\rt)},$$
implies that $\lim\limits_{y\to\infty} \om(y) = 0$. All the properties $(i)-(iv)$ can be directly verified in this case.
\end{proof}

We are now ready to prove Theorem \ref{maintheorem2}. Let $u\in\CORn$, we write $u = v\phi$. For any $\ep > 0$, we have
\begin{align*}
\int_\ep^\infty\int_{\Rn} |\na u|^2 t^s dx dt & = \int_\ep^\infty\int_{\Rn} |\na v|^2 \phi^2 t^s dx dt+ \frac12 \int_\ep^\infty\int_{\Rn} \la \na v^2, \na \phi^2\ra t^s dxdt\\
&\quad + \int_\ep^\infty\int_{\Rn} v^2 |\na\phi|^2 t^s dx dt.
\end{align*}
Using the divergence theorem and the fact that 
$$L_s\phi^2 = 2\phi L_s\phi + 2|\na \phi|^2 = -\frac{\be^2}2 \frac{\phi^2}{x_n^2} + 2|\na\phi|^2$$
we have
\begin{align*}
\int_\ep^\infty\int_{\Rn} |\na u|^2 t^s dx dt & = \int_\ep^\infty\int_{\Rn} |\na v|^2 \phi^2 t^s dx dt+\frac{\be^2}4\int_\ep^\infty\int_{\Rn} \frac{u^2}{x_n^2} t^s dx dt\\
&\quad -\int_{\Rn} u(x,\ep)^2 \ep^s \frac{\pa_t\phi(x,\ep)}{\phi(x,\ep)} dx.
\end{align*}
Since $u \in \CORn$ then there is $\de > 0$ such that $u(x,t) =0$ if $x_n < \de$.  Lemma \ref{solution2} then gives us
$$\lim_{\ep \to 0}\,\ep^s \frac{\pa_t\phi(x,\ep)}{\vphi(x,\ep)} = \frac{k(s,\be)}{x_n^{1-s}},$$
uniformly on $x_n \geq \de$. Letting $\ep$ tend to $0$, we obtain
\begin{align}\label{eq:L2gradient1}
\int_0^\infty\int_{\Rn} |\na u|^2 t^s dx dt & = \frac{\be^2}4\int_0^\infty\int_{\Rn} \frac{u^2}{x_n^2} t^s dx dt + k(s,\be)\int_{\Rn} \frac{u^2}{x_n^{1-s}} dx\notag\\
&\quad\quad\quad\quad + \int_0^\infty\int_{\Rn} |\na v|^2 \phi^2 t^s dx dt.
\end{align}  
The inequality \eqref{eq:mainresult2} is an immediate consequence of \eqref{eq:L2gradient1}. 

We next verify the optimality of $k(s,\be)$. To do this, let us take a smooth function $\eta \in C_0^\infty(\R^{n-1})$ such that $0\leq \eta\leq 1$, $\eta(z) =1$ if $|z| \leq 1$, and $\eta(z) = 0$ if $|z| \geq 2$, and a the function $h\in C^\infty(\R_+)$ such that $0\leq h\leq 1$, $h(y) = 1$ if $0\leq y\leq 1$ and $h(y) =0$ if $y\geq 2$. We first consider the case $\be < 1$. For $\ep > 0$, we define the function, we write $x =(x',x_n)$ with $x'\in \R^{n-1}$,
\begin{equation}\label{eq:assympfunc}
u_\ep(x',x_n,t) = 
\begin{cases}
\eta(x') h(x_n) x_n^{-\frac s2} \om\lt(\frac t{x_n}\rt) &\mbox{if $t\geq \ep$,}\\
\eta(x') h(x_n) x_n^{-\frac s2} \om\lt(\frac \ep{x_n}\rt) &\mbox{if $0\leq t < \ep$.}
\end{cases}
\end{equation}
Let $v_\ep ={u_\ep}/\phi$. It is easy to check that 
$$\int_{\Rn} \frac{u_\ep(x)^2}{x_n^{1-s}} dx\geq \int_{\R^{n-1}} \eta^2 dx' \int_{\ep}^\infty \frac{\om(y)^2}y dy\to \infty,\quad \text{when } \ep \to 0.$$

In the estimates below, we write $C$ for a constant which does not depends on $\ep$ and can be changed from line to line. Then, we have
\begin{align}\label{eq:lonhonep}
\int_\ep^\infty \int_{\Rn} |\na v_\ep|^2 \phi^2\,t^s dxdt & \leq C \int_\ep^\infty\int_0^2\,\om\lt(\frac t{x_n}\rt)^2 x_n^{-s} t^s dx dt\\
&=C\int_0^2 x_n \int_{\frac\ep{x_n}}^{\infty} t^s \om(t)^2 dt dx_n\\
&\leq C. 
\end{align}
We next estimate
\begin{align*}
\int_0^\ep \int_{\Rn} |\na v_\ep|^2 \phi^2\, t^s dxdt.
\end{align*}
The boundedness of $\omega$ yields that
\begin{align*}
\int_0^\ep \int_{\Rn} |\na (h\eta)|^2 \om\lt(\frac \ep{x_n}\rt)^2 x_n^{-s} t^s dx dt \leq C\int_0^2 x_n^{-s} dx_n \int_0^\ep t^s dt \to 0,\quad\text{when } \ep \to 0.
\end{align*}
We also have
\begin{align*}
\int_0^\ep \int_{\Rn} (\eta h)^2 \lt(\om'\lt(\frac\ep{x_n}\rt)\rt)^2\frac{\ep^2}{x_n^4}x_n^{-s} t^s dx dt& \leq C \ep^{s+3} \int_0^2\lt(\om'\lt(\frac\ep{x_n}\rt)\rt)^2\frac{1}{x_n^4}x_n^{-s} dx_n\\
& = C\int_{\frac\ep 2}^\infty (\om'(y))^2 y^{s+2} dy\\
&\leq C.
\end{align*}
Since $y^2(\om'(y))^2 \leq C \om(y)^2$ for any $y > 0$, we can easily check that
$$\int_0^\ep \int_{\Rn}(\eta h)^2 \om\lt(\frac \ep{x_n}\rt)^2 \lt(\frac{\om'\lt(t/x_n\rt)}{\om\lt(t/x_n\rt)}\rt)^2 x_n^{-4-s} t^{2+s} dx dt \leq C.$$
Finally, since
$$\lt(\frac{\om'(y)}{\om(y)}\rt)^2 \leq C y^{-2s} (1+y^2)^{s-1},$$
then
\begin{align*}
\int_0^\ep \int_{\Rn}(\eta h)^2\om\lt(\frac \ep{x_n}\rt)^2 \lt(\frac{\om'\lt(t/x_n\rt)}{\om\lt(t/x_n\rt)}\rt)^2 x_n^{-s-2} t^s dx dt&\leq C \int_0^\ep \int_0^t \om\lt(\frac \ep{x_n}\rt)^2 x_n^{-s} t^{s-2} dx_n dt\\
&\quad + C \int_0^\ep\int_t^\infty t^{-s} x_n^{s-2}\om\lt(\frac \ep{x_n}\rt)^2 dx_n dt\\
&\leq C,
\end{align*}
here, in the last estimate, we use the fact $\om(y) \leq C y^{-(1+s+\sqrt{1-\be^2})/2}$ for $y \geq 1$. Combining the estimates above together, we obtain
$$\int_0^\ep \int_{\Rn} |\na v_\ep|^2 \phi^2\, t^s dxdt \leq C.$$
Therefore
$$\lim_{\ep \to 0} \, \frac{\int_0^\infty\int_{\Rn} |\na u_\ep|^2 t^s dx dt - \frac{\be^2}4\int_0^\infty\int_{\Rn} \frac{u_\ep^2}{x_n^2} t^s dx dt}{\int_{\Rn} \frac{u_\ep(x)^2}{x_n^{1-s}} dx} = k(s,\be).$$
That is $k(s,\be)$ is optimal.

If $\be =1$, we define for $\ep >0, \de >0$,
\begin{equation}\label{eq:assympfunc1}
u_{\ep,\de}(x',x_n,t) = 
\begin{cases}
\eta(x') h(x_n) x_n^{-\frac s2} \lt(\om\lt(\frac t{x_n}\rt)\rt)^{1+\de} &\mbox{if $t\geq \ep$,}\\
\eta(x') h(x_n) x_n^{-\frac s2} \lt(\om\lt(\frac \ep{x_n}\rt)\rt)^{1+\de} &\mbox{if $0\leq t < \ep$.}
\end{cases}
\end{equation}
Let $v_{\ep,\de} = u_{\ep, \de}/ \phi$. Using the similar estimates as above, we have
$$\int_{\Rn} \frac{u_{\ep,\de}(x)^2}{x_n^{1-s}} dx\geq \int_{\R^{n-1}} \eta^2 dx' \int_{\ep}^\infty \frac{\om(y)^{2(1+\de)}}y dy\geq -C\ln \ep,$$
$$\int_\ep^\infty\int_{\Rn} |\na v_{\ep,\de}|^2 \phi^2 t^s dx dt \leq \frac C\de -C\de \ln \ep,$$
and
$$\int_0^\ep\int_{\Rn} |\na v_{\ep,\de}|^2 \phi^2 t^s dx dt \leq \frac C\de,$$
where $C$ is a constant which does not depend on $\ep,\de$. Hence
$$\lim_{\de\to 0}\lim_{\ep\to 0}\, \frac{\int_0^\infty\int_{\Rn} |\na u_{\ep,\de}|^2 t^s dx dt - \frac{1}4\int_0^\infty\int_{\Rn} \frac{u_{\ep,\de}^2}{x_n^2} t^s dx dt}{\int_{\Rn} \frac{u_{\ep,\de}(x)^2}{x_n^{1-s}} dx} = k(s,1),$$ 
that is $k(s,1)$ is optimal.

Let us prove the inequality \eqref{eq:improvementversion2}. Its proof is completely similar with the proof of part $(iii)$ of Theorem $2$ in \cite{FMT13}. Lemma \ref{solution2} yields that
\begin{equation}\label{eq:aaaa}
t^{\frac s2} \phi^{\frac{2n+s}{n+s-1}} \sim \frac{t^{\frac s2} x_n^{\frac{2n+s}{n+s-1} \frac{1+\sqrt{1-\be^2}}2}}{(x_n^2 + t^2)^{\frac{2n+s}{n+s-1}\frac{1+s+\sqrt{1-\be^2}}4}},
\end{equation}
and
\begin{equation}\label{eq:aaaa1}
t^{\frac s2} \phi^{\frac{n+1}{n+s-1}}|\na \phi| \sim
\begin{cases}
\frac{t^{\frac s2} x_n^{\frac{2n+s}{n+s-1} \frac{1+\sqrt{1-\be^2}}2-1}}{(x_n^2 + t^2)^{\frac{2n+s}{n+s-1}\frac{1+s+\sqrt{1-\be^2}}4}}&\mbox{ if $s\in (-1,0]$,}\\
\frac{t^{-\frac s2} x_n^{\frac{2n+s}{n+s-1} \frac{1+\sqrt{1-\be^2}}2-1}}{(x_n^2 + t^2)^{\frac{2n+s}{n+s-1}\frac{1+s+\sqrt{1-\be^2}}4 -\frac s2}}&\mbox{ if $s\in (0,1)$.}
\end{cases}
\end{equation}
From \eqref{eq:L2gradient1}, it is enough to prove that
\begin{equation}\label{eq:enoughsobolev}
\int_0^\infty\int_{\Rn} \lt|\na u -u\frac{\na\phi}\phi\rt|^2t^s dx dt\geq c\lt(\int_{\R^n_+} |u(x,0)|^{\frac{2n}{n-1+s}} dx\rt)^{\frac{n-1+s}n}.
\end{equation}
Our starting point is the following Sobolev type inequalities \cite[Theorem $1$, section $2.1.6$]{Maz11}
\begin{equation}\label{eq:abba}
\int_0^\infty\int_{\Rn} t^{\frac s2} |\na u| dxdt \geq c\lt(\int_0^\infty\int_{\Rn} |u(x,t)|^{\frac{2(n+1)}{2n+s}} dxdt\rt)^{\frac{2n+s}{2(n+1)}}
\end{equation}
and
\begin{equation}\label{eq:baab}
\int_0^\infty\int_{\Rn} t^{\frac s2} |\na u| dxdt \geq c\lt(\int_{\Rn} |u(x,0)|^{\frac{2n}{2n+s}} dx\rt)^{\frac{2n+s}{2n}}
\end{equation}
for any $u\in C_0^\infty(\Rn\times \R)$. Applying these inequalities for $u = \phi^{\frac{2n+s}{n+s-1}} v$, we obtain
\begin{align}\label{eq:cddc}
\int_0^\infty\int_{\Rn} &t^{\frac s2}  \phi^{\frac{2n+s}{n+s-1}}|\na v| dxdt +\frac{2n+s}{n+s-1}\int_0^\infty\int_{\Rn} t^{\frac s2}  \phi^{\frac{n+1}{n+s-1}}|\na \phi| |v|dxdt\notag\\
&\geq c\lt(\int_0^\infty\int_{\Rn}\phi^{\frac{2n+s}{n+s-1}} |v|^{\frac{2(n+1)}{2n+s}} dxdt\rt)^{\frac{2n+s}{2(n+1)}}
\end{align}
and
\begin{align}\label{eq:dccd}
\int_0^\infty\int_{\Rn} &t^{\frac s2}  \phi^{\frac{2n+s}{n+s-1}}|\na v| dxdt +\frac{2n+s}{n+s-1}\int_0^\infty\int_{\Rn} t^{\frac s2}  \phi^{\frac{n+1}{n+s-1}}|\na \phi| |v|dxdt\notag\\
&\geq c\lt(\int_{\Rn} \phi(x,0)^{\frac{2n+s}{n+s-1}}|u(x,0)|^{\frac{2n}{2n+s}} dx\rt)^{\frac{2n+s}{2n}}.
\end{align}
Denote
\[
A =\frac s2,\quad B = \frac{2n+s}{n+s-1} \frac{1+\sqrt{1-\be^2}}2-1,\quad \Gamma = \frac{2n+s}{n+s-1}\frac{1+s+\sqrt{1-\be^2}}4.
\]
A simple computation shows that
\[
A+ B +2-2\Gamma = \frac{(2-s)(n-1)}{2(n+s-1)} >0.
\]
Hence by \cite[Lemma $5$]{FMT13} if $s\in (-1,0]$ and \cite[Lemma $11$]{FMT13} if $s\in (0,1)$, and by \eqref{eq:aaaa}, \eqref{eq:aaaa1} we obtain for any $s\in (-1,1)$
\begin{equation}\label{eq:bbbb}
\int_0^\infty\int_{\Rn} t^{\frac s2}  \phi^{\frac{2n+s}{n+s-1}}|\na v| dxdt \geq c\lt(\int_0^\infty\int_{\Rn}\phi^{\frac{2n+s}{n+s-1}} |v|^{\frac{2(n+1)}{2n+s}} dxdt\rt)^{\frac{2n+s}{2(n+1)}}
\end{equation}
and
\begin{equation}\label{eq:bbbb1}
\int_0^\infty\int_{\Rn} t^{\frac s2}  \phi^{\frac{2n+s}{n+s-1}}|\na v| dxdt \geq c\lt(\int_{\Rn} \phi(x,0)^{\frac{2n+s}{n+s-1}}|u(x,0)|^{\frac{2n}{2n+s}} dx\rt)^{\frac{2n+s}{2n}}.
\end{equation}
Setting $v = w^{\frac{2n+s}{n+s-1}}$ in \eqref{eq:bbbb} and applying Cauchy-Schwartz inequality, we arrrive
\begin{equation}\label{eq:cccc}
\int_0^\infty\int_{\Rn}|\na w|^2 \phi^2 t^s dtdx \geq c\lt(\int_0^\infty\int_{\Rn} |\phi w|^{\frac{2(n+1)}{2n+s}} dxdt\rt)^{\frac{2n+s}{2(n+1)}}.
\end{equation}
Setting $v = w^{\frac{2n+s}{n+s-1}}$ in \eqref{eq:bbbb} and applying Cauchy-Schwartz inequality and \eqref{eq:cccc}, we arrrive the inequality
\begin{equation}\label{eq:cccc1}
\int_0^\infty\int_{\Rn}|\na w|^2 \phi^2 t^s dtdx \geq c\lt(\int_{\Rn} |\phi(x,0) w(x,0)|^{\frac{2n}{n+s-1}} dx\rt)^{\frac{n+s-1}{2n}}
\end{equation}
which is equivalent to \eqref{eq:enoughsobolev}. Theorem \ref{maintheorem2} is completely proved.

We continue by proving Theorem \ref{maintheorem3}. For this purpose, let us define 
$$\phi(x,t) = d(x)^{-\frac s2} \om\lt(\frac t{d(x)}\rt),$$
where $\om$ solves the equation \eqref{eq:EL2}. A straightforward computation shows that
\begin{equation}\label{eq:xxxxx}
L_s \phi = -\frac{\be^2}4 \frac{\phi}{d(x)^2}+ \frac{-\De d} d \phi \frac{\frac s2 \om\lt(\frac t{d(x)}\rt) + \frac td \om'\lt(\frac t{d(x)}\rt)}{\om\lt(\frac t{d(x)}\rt)} =-\frac{\be^2}4 \frac{\phi}{d(x)^2} + t^{-s}\text{div}(t^s\na \phi).
\end{equation}
By the same way to \eqref{eq:L2gradient1}, we arrive
\begin{align}\label{eq:L2gradient2}
\int_0^\infty\int_{\Om} |\na u|^2 t^s dx dt & = \frac{\be^2}4\int_0^\infty\int_{\Om} \frac{u^2}{d(x)^2} t^s dx dt + k(s,\be)\int_{\Om} \frac{u(x,0)}{d(x)^{1-s}} dx\notag\\
&\quad + \int_0^\infty\int_{\Om} \lt|\na u -u\frac{\na \phi}{\phi}\rt|^2 t^s dx dt-\int_0^\infty\int_{\Om} \frac{\text{div}(t^s\phi)}{\phi} u^2 dxdt,
\end{align}
which implies \eqref{eq:mainresult3} since $-\text{div}(t^s\phi) \geq 0$ by by Lemma \ref{solution2} and the assumption $s\in (-1,0]$.

To check the optimality of $k(s,\be)$, we can assume that $x_0 =0$ by translate in the $\R^n$. If the boundary of $\Om$ is flat at $0$, that is, there is $r > 0$ such that after a change of coordinate, we have
$$B(0,r) \cap \Om = \{x_n > 0\} \cap B(0,r),\quad B(0,r)\cap \pa \Om =\{x_n=0\} \cap B(0,r).$$
In this case, we can exploit the argument in proof of Theorem \ref{maintheorem2} to verify the optimality of $k(s,\be)$. In the general case where $\pa\Om$ is not flat at $0$, we can use the same way in the proof of Theorem $1$ in \cite[page $124$]{FMT13}. 

It remains to prove \eqref{eq:improvementversion3}. By \eqref{eq:L2gradient2}, it suffices to prove
\begin{align}\label{eq:yyyy}
\int_0^\infty\int_{\Om} \lt|\na u -u\frac{\na \phi}{\phi}\rt|^2 t^s dx dt -\int_0^\infty\int_{\Om} \frac{\text{div}(t^s\phi)}{\phi} u^2 dxdt \geq c\lt(\int_\Om |u(x,0)|^{\frac{2n}{(n-1+s)}} dx\rt)^{\frac{n-1+s}n}.
\end{align}

We remark that inequalities \eqref{eq:cddc} and \eqref{eq:dccd} sill hold if we replace $\Rn$ by $\Om$. Since $s\in (-1,0)$, Lemma \ref{solution2} gives us the following estimates
$$-\phi L_s\phi - \frac{\be^2}4 \frac{\phi^2}{d^2} \sim (-\De d) d^{\sqrt{1-\be^2}} (d^2 +t^2)^{-\frac{1+s+\sqrt{1-\be^2}}2},$$
$$t^{\frac s2} \phi^{\frac{2n+s}{n+s-1}} \sim \frac{t^{\frac s2} d^{\frac{2n+s}{n+s-1} \frac{1+\sqrt{1-\be^2}}2}}{(d^2 + t^2)^{\frac{2n+s}{n+s-1}\frac{1+s+\sqrt{1-\be^2}}4}},$$
and
$$t^{\frac s2} \phi^{\frac{n+1}{n+s-1}}|\na \phi| \sim \frac{t^{\frac s2} d^{\frac{2n+s}{n+s-1} \frac{1+\sqrt{1-\be^2}}2-1}}{(d^2 + t^2)^{\frac{2n+s}{n+s-1}\frac{1+s+\sqrt{1-\be^2}}4}}.$$
Denote again
\[
A =\frac s2,\quad B = \frac{2n+s}{n+s-1} \frac{1+\sqrt{1-\be^2}}2-1,\quad \Gamma = \frac{2n+s}{n+s-1}\frac{1+s+\sqrt{1-\be^2}}4.
\]
Since $A +B+2 -2\Gamma >0$, by \cite[Lemma $7$]{FMT13} and previous asymptotic estimates, we get
\begin{align*}
\int_0^\infty \int_\Om \frac{t^{\frac s2} d^{\frac{2n+s}{n+s-1} \frac{1+\sqrt{1-\be^2}}2-1}}{(d^2 + t^2)^{\frac{2n+s}{n+s-1}\frac{1+s+\sqrt{1-\be^2}}4}} |v| dxdt &\leq C_1 \int_0^\infty \int_\Om\frac{t^{\frac s2} d^{\frac{2n+s}{n+s-1} \frac{1+\sqrt{1-\be^2}}2}}{(d^2 + t^2)^{\frac{2n+s}{n+s-1}\frac{1+s+\sqrt{1-\be^2}}4}} |\na v| dxdt\\
&\quad + C_2\int_0^\infty \int_\Om\frac{t^{\frac s2} d^{\frac{2n+s}{n+s-1} \frac{1+\sqrt{1-\be^2}}2}}{(d^2 + t^2)^{\frac{2n+s}{n+s-1}\frac{1+s+\sqrt{1-\be^2}}4}} | v| dxdt.
\end{align*}
From this inequality and \eqref{eq:cddc}, \eqref{eq:dccd} we have that
\begin{align}\label{eq:cddc1}
\int_0^\infty\int_{\Om} t^{\frac s2}  \phi^{\frac{2n+s}{n+s-1}}|\na v| dxdt +\int_0^\infty\int_{\Om} t^{\frac s2}  \phi^{\frac{2n+s}{n+s-1}} |v|dxdt \geq c\lt(\int_0^\infty\int_{\Om}\phi^{\frac{2n+s}{n+s-1}} |v|^{\frac{2(n+1)}{2n+s}} dxdt\rt)^{\frac{2n+s}{2(n+1)}}
\end{align}
and
\begin{align}\label{eq:dccd1}
\int_0^\infty\int_{\Omega} t^{\frac s2}  \phi^{\frac{2n+s}{n+s-1}}|\na v| dxdt +\int_0^\infty\int_{\Omega} t^{\frac s2}  \phi^{\frac{2n+s}{n+s-1}}|v|dxdt\geq c\lt(\int_{\Omega} \phi^{\frac{2n+s}{n+s-1}}|u|^{\frac{2n}{2n+s}} dx\rt)^{\frac{2n+s}{2n}}.
\end{align}
Taking $v= |w|^{\frac{2n+s}{n+s-1}}$ into \eqref{eq:cddc1}, and applying Cauchy-Schwartz inequality in the left-hand side, and then making simplification, we arrive 
\begin{equation}\label{eq:123321}
\int_0^\infty\int_{\Omega} \phi^2 |\na w|^2 t^s dxdt + \int_0^\infty\int_{\Omega} \phi^2 |\na w|^2 t^s dxdt \geq  c\lt(\int_0^\infty\int_{\Om}|\phi w|^{\frac{2(n+1)}{n+s-1}}  dxdt\rt)^{\frac{n+s-1}{2(n+1)}}.
\end{equation}
Since $\Om$ has a finite inner radius, using \cite[Lemma $10$]{FMT13}, we obtain 
\begin{align}\label{eq:a11}
c\int_0^\infty\int_{\Omega} \phi^2 |\na w|^2 t^s dxdt \leq \int_0^\infty\int_{\Omega} \phi^2 |\na w|^2 t^s dxdt -\int_0^\infty\int_{\Omega} \text{div}(t^s \na \phi) \phi w^2 dxdt,
\end{align}
which then implies 
\begin{equation}\label{eq:2121}
\int_0^\infty\int_{\Omega} \phi^2 |\na w|^2 t^s dxdt -\int_0^\infty\int_{\Omega} \text{div}(t^s \na \phi) \phi w^2 dxdt \geq c\lt(\int_0^\infty\int_{\Om}|\phi w|^{\frac{2(n+1)}{n+s-1}}  dxdt\rt)^{\frac{n+s-1}{2(n+1)}}.
\end{equation}
Taking $v= |w|^{\frac{2n+s}{n+s-1}}$ into \eqref{eq:dccd1}, applying Cauchy-Schwartz inequality, inequalities \eqref{eq:a11}, \eqref{eq:2121} and thereafter making simplification, we arrive the inequality
\[
\int_0^\infty\int_{\Omega} \phi^2 |\na w|^2 t^s dxdt -\int_0^\infty\int_{\Omega} \text{div}(t^s \na \phi) \phi w^2 dxdt \geq c\lt(\int_\Om |(\phi w)(x,0)|^{\frac{2n}{(n-1+s)}} dx\rt)^{\frac{n-1+s}n},
\]
which is equivalent to \eqref{eq:yyyy}. Theorem \ref{maintheorem3} is then completely proved.

In the special case where $\be =1$, $s =0$, Theorem \ref{maintheorem2} and Theorem \ref{maintheorem3} yield an improved Hardy inequality on the half spaces and on tubes,
\begin{corollary}\label{improHardy}
Given $n\geq 2$, then we have 
\begin{description}
\item (i) For any $u\in \CORn$, it holds
$$\int_\R\int_{\Rn} |\na u(x,t)|^2 dx dt \geq \frac14 \int_\R\int_{\Rn} \frac{u(x,t)^2}{x_n^2} dx dt + 4\lt(\frac{\Gam\lt(\frac34\rt)}{\Gam\lt(\frac14\rt)}\rt)^2 \int_{\Rn}\frac{u(x,0)^2}{x_n} dx.$$
\item (ii) Let $\Om$ be a domain in $\R^n$ satisfying \eqref{eq:meanconvex}. Then for any $u\in \COR$, it holds
$$\int_\R\int_{\Om} |\na u(x,t)|^2 dx dt \geq \frac14 \int_\R\int_{\Om} \frac{u(x,t)^2}{d(x)^2} dx dt + 4\lt(\frac{\Gam\lt(\frac34\rt)}{\Gam\lt(\frac14\rt)}\rt)^2 \int_{\Om}\frac{u(x,0)^2}{d(x)} dx.$$
\end{description}
\end{corollary}
We emphasize that the inequalities in Corollary \ref{improHardy} are an improved version of Hardy inequality on the half spaces and on the tubes with the remainder term concerning to the trace of function on the hyperplanes. This improved version seems to be new. We refer reader to \cite{FTT09} for many other improved Hardy inequalities on the half spaces.

Let $\Om$ be a bounded domain of $\R^n$, and let $\lam_i$ and $\phi_i$ be the Dirichlet eigenvalues and orthonormal eigenfunctions of the Laplacian, that is, $-\De \phi_i = \lam_i \phi_i$ in $\Om$ with $\phi_i =0$ on $\pa\Om$. Then the spectral fractional Laplacian on $\Om$ is defined for any $f=\sum c_i\phi_i$ by
$$(-\De)^{\al}f(x) = \sum_{i=1}^\infty c_i \lam_i^\al \phi_i(x),\quad \al \in (0,1).$$
It is shown in \cite[Appendix $8.1$]{FMT13} that if 
\begin{equation}\label{eq:EXTEND}
u(x,t) = \sum_{i=1}^\infty c_i \phi_i(x) T(\sqrt{\lam_i} t),
\end{equation}
where 
$$T(t) = \frac{2^{1-\al}}{\Gam(\al)} t^\al K_\al(t),$$
with $K_\al$ is the modified Bessel function of the second kind, then it holds that $u(x,0) = f(x)$, $u(x,y) =0$ in $\pa\Om\times (0,\infty)$, $\d(y^{1-2\al}\na u) =0$ in $\COR$, and
\begin{equation}\label{eq:fractionalLap}
\int_0^\infty \int_\Om |\na u|^2 t^{1-2\al} dx dt = \frac{2^{1-2\al}\Gam(1-\al)}{\Gam(\al)} \la (-\De)^\al f, f\ra.
\end{equation}
Moreover, $u$ minimizes the energy
$$\i0i\int_\Om |\na v|^2 t^{1-2\al} dx dt,$$
over all functions $v$ such that $v(x,0) = f(x)$ and $v(x,y) =0$ in $\pa\Om\times (0,\infty)$.
 
As a consequence of Theorem \ref{maintheorem3}, we have the following Hardy type inequality and Hardy-Sobolev-Maz'ya type inequality for spectral fractional Laplacian.
\begin{corollary}
Let $n\ge 2$, $1/2 \leq \al < 1$, $\be \in [0,1]$, and let $\Om\subset \R^n$ be a bounded domain and satisfy the condition \eqref{eq:meanconvex}. Then
\begin{description}
\item (i) For any $f\in C_0^\infty(\Om)$ we have
\begin{align*}
\la (-\De)^\al f, f\ra \geq \frac{\be^2}4\frac{2^{2\al-1}\Gam(\al)}{\Gam(1-\al)}\int_0^\infty \int_\Om \frac{u(x,t)^2}{d(x)^2} dx dt + 2^{2\al}\frac{\Gam\lt(\frac{2(1+\al)+\sqrt{1-\be^2}}4\rt)^2}{\Gam\lt(\frac{2(1-\al)+\sqrt{1-\be^2}}4\rt)^2} \int_\Om \frac{f(x)^2}{d(x)^{2\al}} dx. 
\end{align*}
\item (ii) Suppose that there exist $x_0\in \pa \Om$ and $r > 0$ such that the boundary $\pa \Om \cap B(x_0,r)$ is $C^1-$regular, then the constant $2^{2\al} \lt(\frac{\Gam((2(1+\al) +\sqrt{1-\be^2})/4)}{\Gam((2(1-\al) +\sqrt{1-\be^2})/4)}\rt)^2$ is optimal. 
\item (iii) If $\Om$ is Lipschitz and $\al \in (1/2,1)$, then there exists a constant $c>0$ such that
\begin{align*}
\la (-\De)^\al f, f\ra &\geq \frac{\be^2}4\frac{2^{2\al-1}\Gam(\al)}{\Gam(1-\al)}\int_0^\infty \int_\Om \frac{u^2}{d^2} dx dt + 2^{2\al}\frac{\Gam\lt(\frac{2(1+\al)+\sqrt{1-\be^2}}4\rt)^2}{\Gam\lt(\frac{2(1-\al)+\sqrt{1-\be^2}}4\rt)^2} \int_\Om \frac{f^2}{d^{2\al}} dx\\
&\quad\quad\quad\quad + c\lt(\int_\Om |f|^{\frac{2n}{n-2\al}} dx\rt)^{\frac{n-2\al}n}.
\end{align*}
\end{description}
\end{corollary}
\begin{proof}
The part $(i)$ and $(iii)$ are immediately derived from Theorem \ref{maintheorem3} and \eqref{eq:fractionalLap} with $s = 1-2\al$. For the part $(ii)$, we argue as follows. Given $\ep > 0$, choose a function $v\in \COR$ such that 
$$k(s,\be) + \ep > \frac{\i0i\int_\Om |\na v|^2 t^s dx dt -\frac{\be^2}4\i0i\int_\Om \frac{u^2}{d(x)^2} t^s dx dt}{\int_\Om \frac{u(x,0)^2}{d(x)^{1-s}} dx}.$$
Let $f(x) = v(x,0)$ and extend it to $\Om\times (0,\infty)$ by \eqref{eq:EXTEND}, then we have
$$\la (-\De)^\al f, f\ra \leq 2^{2\al-1} \frac{\Gam(\al)}{\Gam(1-\al)} \i0i\int_\Om |\na v|^2 t^s dx dt.$$
Combining two inequalities above we get the part $(ii)$.
\end{proof}

\section{Proof of Theorem \ref{logHS} and Theorem \ref{Radialcase}}
In this section, we prove Theorem \ref{logHS} and Theorem \ref{Radialcase}. As we see below, Theorem \ref{logHS} is a simple consequence of \eqref{eq:traceSobolev}, \eqref{eq:traceHardy}, and H\"older inequality.

\emph{Proof of Theorem \ref{logHS}:} We first prove the logarithmic Sobolev trace inequality. Let $u \in \dot{W}(d_{\hs},\hs)$ such that $\int_{\sp}u(x,0)^2 dx =1$. For any $2 < p< 2(s)^*$, it implies from H\"older inequality and the weighted trace Sobolev inequality \eqref{eq:traceSobolev} that  
\begin{align*}
\int_{\sp} |u(x,0)|^p dx &\leq \lt(\int_{\sp} |u(x,0)|^{2(s)^*} dx\rt)^{\frac{p-2}{2(s)^*-2}}\\
&\leq \lt(C_{n,s} \int_{\hs} |\na u(x,t)|^2 t^s dxdt\rt)^{\frac{2(s)^*(p-2)}{2(2(s)^*-2)}}.
\end{align*}
Taking logarithmic and deviding both sides by $p-2$, we obtain
$$\frac1{p-2}\ln\lt(\int_{\sp} |u(x,0)|^p dx\rt)\leq \frac{n}{2(1-s)}\ln\lt(C_{n,s} \int_{\hs} |\na u(x,t)|^2 t^s dxdt\rt).$$
Let $p$ tend to $2$ we obtain the logarithmic Sobolev trace inequality \eqref{eq:logSobtrace} with $C_1 =C_{n,s}$.

The logarithmic Hardy trace inequality \eqref{eq:logHardytrace} is proved by the same way. For any $2 < p< 2(s)^*$, we denote
$$\ga_p = \frac{1-s}{p}-\frac{(p-2)(n+s-1)}{2p}.$$
Using H\"older inequality, we have
$$\int_{\sp}\frac{|u(x,0)|^p}{|x|^{\ga_p p}} dx \leq \lt(\int_{\sp} |u(x,0)|^{2(s)^*} dx\rt)^{\frac{p-2}{2(s)^*-2}}.$$
Applying the weighted trace Sobolev inequality, and using the same argument in the proof of the logarithmic Sobolev trace inequality, we obtain the logarithmic Hardy trace inequality \eqref{eq:logHardytrace}. The proof of Theorem \ref{logHS} hence is completed.

{\bf Proof of Theorem \ref{Radialcase}:} Since $u$ is radial, then $u(x,0)$ is also radial on $\sp$. Suppose that $u(x,t) =v(|(x,t)|)$ with a function $v: \R_+ \to \R$. Using the spherical coordinate, we have
$$\int_{\hs} |\na u|^2 t^s dx dt =\om_{n,s} \int_0^\infty |v'(r)|^2 r^{n_s-1} dr,$$
with
$$\om_{n,s} =\int_{\{|(x,t)| =1, t\geq 0\}}t^s d\mH^n =\frac{\pi^{\frac n2} \Gam\lt(\frac{1+s}{2}\rt)}{\Gam\lt(\frac{n_s}{2}\rt)},$$
where $\mH^n$ is $n-$dimensional Hausdorff measure on the unit sphere $S^n$.

We calculate $C_{LS,r}$ first. It is evident that
$$\int_{\sp} u(x,0)^2 \ln u(x,0)^2 dx = \om_{n-1} \int_0^\infty v(r)^2 \ln\lt( v(r)^2 \rt) r^{n-1} dr,$$
with $\om_{n-1} = 2\pi^{n/2}/\Gam\lt(n/2\rt)$ is the surface area of the unit sphere $S^{n-1}$. Hence
$$C_{LS,r}(n,s) =\sup\lt\{\frac{\exp\lt(\frac{1-s}{n}\om_{n-1}\int_0^\infty v(r)^2\ln (v(r)^2) r^{n-1}dr\rt)}{\om_{n,s}\int_0^\infty (v'(r)^2) r^{n_s-1}dr}\, :\, \om_{n-1} \int_0^\infty v(r)^2 r^{n-1} dr =1\rt\}.$$
After changing the variable $r$ by $r^{2/(1-s)}$, we arrive
\begin{equation}\label{eq:sup}
C_{LS,r}(n,s) =\frac{4\om_{n-1}}{(1-s)^2\om_{n,s}} \lt(\frac{2\om_{n-1}}{1-s}\rt)^{-\frac{1-s}{n}}\sup\,\frac{\exp\lt(\frac{1-s}{n}\int_0^\infty v(r)^2\ln (v(r)^2) r^{\frac{2n}{1-s}-1}dr\rt)}{\int_0^\infty (v'(r)^2) r^{\frac{2n}{1-s}-1}dr},
\end{equation}
where the supremum is taken on set of functions $v$ such that 
$$\int_0^\infty v(r)^2 r^{\frac{2n}{1-s}-1} dr =1.$$
We next compute the supremum in \eqref{eq:sup}. It is given in the following lemma.
\begin{lemma}\label{logsobfracdimen}
Given $a \geq 1$. For any function $u$ on $\R_+$ such that $\int_{\R+} (u'(r))^2 r^{a-1} dr < \infty$ and $\int_{\R_+} u(r)^2 r^{a-1}dr =1$ then
\begin{equation}\label{eq:dpcm}
\int_{\R^+}u(r)^2 \log(u(r)^2)\, r^{a-1} dr \leq \frac{a}2 \ln\lt(\frac{2}{ae}\lt(\frac{2}{\Gam\lt(\frac a2\rt)}\rt)^{\frac2a}\int_{\R_+} (u'(r))^2 r^{a-1} dr\rt).
\end{equation}
The equality holds if 
\begin{equation}\label{eq:extremal}
u(r) = \lam^{\frac a2} \lt(\frac{2}{\Gam\lt(\frac a2\rt)}\rt)^{\frac 12}\, e^{-\frac{\lam^2 r^2}{2}},
\end{equation}
for some $\lam >0$.
\end{lemma}
\begin{proof}
We follow the argument in \cite{DC02} which is based on the mass transport method to prove this lemma. Approximating $u$ if necessary, we can assume that $u$ is strictly positive and continuous. Consider the function
$$u_0(r) = \frac1{c_a} e^{-\frac{r^2}{2}},$$
where $c_a = \lt(\Gam\lt(a/2\rt)/{2}\rt)^{\frac 12}$ is normalized constant such that $\int_0^\infty u_0^2 r^{a-1} dr =1$. For each $r >0$, define $T(r)$ is the unique number such that
$$\int_0^r u(t)^2 t^{a-1} dt = \int_0^{T(r)} u_0(t)^2 t^{a-1} dt.$$
The function $T$ is strictly increasing, $T(0) =0$, $\lim\limits_{r\to\infty} T(r) =\infty$, and satisfies the equation
$$u(r)^2 = u_0(T(r))^2 \frac{T(r)^{a-1}}{r^{a-1}} T'(r).$$
Taking the logarithmic both sides and using the Geometric-Arithmetic mean inequality and the inequality $\ln(1+x) \leq x$, we have (note that $T'(r) >0$)
$$\ln(u(r)^2 \leq \ln (u_0(T(r))^2) + T'(r) + (a-1) \frac{T(r)}{r} -a.$$
Multiplying both sides with $u^2(r) r^{a-1}$, integrating on $\R_+$, and using the integration by parts, we obtain
\begin{align*}
\int_0^\infty u(r)^2 \ln(u(r)^2) r^{n-1} dr &\leq \int_0^\infty u(r)^2 r^{a-1}\ln (u_0(T(r))^2) dr\\ 
&\quad -2\int_0^\infty u'(r) u(r) T(r) r^{a-1} dr -a\\
&=\int_0^\infty u_0(T(r))^2 T(r)^{a-1} T'(r)\ln (u_0(T(r))^2)\\
&\quad -2\int_0^\infty u'(r) u(r) T(r) r^{a-1} dr -a\\
&\leq \int_0^\infty u_0(r)^2 r^{a-1} \ln(u_0(r)^2) dr -a\\&\quad \int_0^\infty u'(r)^2 r^{a-1}dr + \int_0^\infty u(r)^2 r^{a-1} T(r)^2 dr\\
&=\int_0^\infty u'(r)^2 r^{a-1}dr -(\ln c_a^2 +a). 
\end{align*}
Changing function $u$ by $u_\ep(r) =\ep^{{a}/{2}} u(\ep r)$, we have
$$\int_0^\infty u(r)^2 \ln(u(r)^2) r^{n-1} dr \leq -a\ln \ep +\ep^2\int_0^\infty u'(r)^2 r^{a-1}dr -(\ln c_a^2 +a).$$
Optimizing the right hand side over $\ep >0$, we get \eqref{eq:dpcm}. Finally, an easy computation that the equality in \eqref{eq:dpcm} holds if $u$ is given by \eqref{eq:extremal}.
\end{proof}
Combining \eqref{eq:sup} and \eqref{eq:dpcm} we obtain the value of $C_{LS,r}(n,s)$ as in \eqref{eq:CLSr}, and the extremal functions are given by
$$u(x,t) = \lam^{\frac n2} \lt(\frac{\Gam\lt(\frac n{1-s}\rt)\om_{n-1}}{1-s}\rt)^{\frac12} \exp\lt(-\frac{|\lam(x,t)|^{1-s}}{2}\rt),$$
for some $\lam > 0$.

By the same argument, we arrive
\begin{equation}\label{eq:sup1}
C_{LH,r}(n,s) =\frac{4\om_{n-1}}{(1-s)^2\om_{n,s}} \lt(\frac{2\om_{n-1}}{1-s}\rt)^{-\frac{1-s}{n}}\sup\,\frac{\exp\lt(\frac{1-s}{n}\int_0^\infty \frac{v(r)^2}{r^2}\ln \lt(\frac{v(r)^2}{r^{2 -\frac{2n}{1-s}}}\rt) r^{\frac{2n}{1-s}-1}dr\rt)}{\int (v'(r)^2) r^{\frac{2n}{1-s}-1}dr},
\end{equation}
where the supremum is taken on set of functions $v$ such that 
$$\int_{\R_+}\frac{v(r)^2}{r^2} r^{\frac{2n}{1-s}-1} dr =1.$$
It follows from Theorem $\mathrm{\bf B'}$ in \cite{DDFT10} that
\begin{lemma}\label{fLHS}
Let $n\geq 2$ and $-1 < a <(n-2)/{2}$. If $\int_0^\infty {u(r)}^2 r^{\frac{n}{1+a}-3} dr =1$, then
\begin{equation}\label{eq:fLHS}
\int_0^\infty \frac{u(r)^2}{r^2} \ln\lt(\frac{u(r)^2}{r^{2-\frac{n}{1+a}}}\rt) r^{\frac{n}{1+a}-1} dr \leq \frac{n}{2(1+a)}\ln\lt(C(n,a) \int_0^\infty (u'(r))^2 r^{\frac{n}{1+a}-1} dr\rt),
\end{equation}
with
$$C(n,a) =\frac{4(1+a)^2}{n} \lt(\frac{1}{2\pi e(1+a)}\rt)^{\frac{1+a}{n}} \lt(\frac{n-1-a}{(n-2(a+1))^2}\rt)^{1-\frac{1+a}{n}}.$$
Moreover, equality in \eqref{eq:fLHS} holds if
$$u(r) =\lt(\frac{\lt(\frac{n}{1+a} -2\rt)^2}{2\lt(\frac{n}{1+a}-1\rt)\pi}\rt)^{\frac14}\lam^{\frac{n-2(1+a)}{2(1+a)}} (\lam r)^{\frac{-n}{2(1+a)}+1} \exp\lt(-\frac{\lt(\frac{n}{1+a} -2\rt)^2}{4\lt(\frac{n}{1+a}-1\rt)}\lt(\ln(\lam r)\rt)^2\rt),$$
for some $\lam > 0$.
\end{lemma}
\begin{proof}
This lemma is a consequence of Theorem $\mathrm{\bf B'}$ in \cite{DDFT10} with $\ga = \frac{n}{4(1+a)}$.
\end{proof}
Let us continue proving Theorem \ref{Radialcase}. Equality \eqref{eq:CHSr} is consequence of\eqref{eq:sup1} and \eqref{eq:fLHS} for the case $a =-(1+s)/{2}$. Moreover, the extremal functions are given by
$$u(x,t) = \lt(\frac{(n-1+s)^2}{2\pi \lt(\frac{2n}{1-s}-1\rt) \om_{n-1}^2}\rt)^{\frac14}\lam^{\frac{n-1+s}{2}}|\lam(x,t)|^{-\frac{n-1+s}2}\exp\lt(\frac{(n-1+s)^2}{4\lt(\frac{2n}{1-s}-1\rt)}\lt(\ln |\lam (x,t)|\rt)^2\rt),$$
for some $\lam >0$.

\section*{Acknowledgments}
The author would like to thank anonymous referees for many valuable suggestions and comments which improve the presentation of this paper.

\end{document}